\documentclass[10pt,a4paper]{amsart}
\usepackage{amsmath, amssymb, amsfonts, amsthm, float, stmaryrd, epsfig}
\usepackage[abbrev,alphabetic]{amsrefs}
\usepackage[pdfborder={0 0 0 [0 0 ]},bookmarksdepth=3, bookmarksopen=true,unicode=true]{hyperref}
\usepackage[latin1,utf8]{inputenc}
\usepackage{textalpha}
\usepackage[capitalize]{cleveref}
\usepackage{color}
\usepackage[ps,all,arc,rotate]{xy}
\usepackage{bbm} 

\usepackage{stackengine,scalerel}
\stackMath
\def\lhdslant{\ThisStyle{\mathrel{%
  \stackinset{r}{.75pt+.15\LMpt}{t}{.1\LMpt}{\rule{.3pt}{1.1\LMex+.2ex}}{\SavedStyle\leqslant}%
}}}

\hypersetup{
    colorlinks=true,
    linkcolor=black,
    citecolor=black,
    filecolor=black,
    urlcolor=black,
}

\numberwithin{figure}{section}
\numberwithin{table}{section}

\theoremstyle{plain}
\newtheorem{thm}{Theorem}[section]
\crefname{thm}{Theorem}{Theorems}
\newtheorem*{prop*}{Proposition}
\newtheorem*{thm*}{Theorem}
\newtheorem{prop}[thm]{Proposition}
\crefname{prop}{Proposition}{Propositions}
\newtheorem{lem}[thm]{Lemma}
\crefname{lem}{Lemma}{Lemmata}

\crefname{cor}{Corollary}{Corollaries}

\crefname{conj}{Conjecture}{Conjectures}
\crefname{equation}{Equation}{Equations}

\theoremstyle{definition}

\newtheorem{dfn}[thm]{Definition}
\newtheorem*{dfn*}{Definition}

\theoremstyle{remark}
\newtheorem{rmk}[thm]{Remark}

\newtheoremstyle{maintheorem}{}{}{\itshape}{}{\bfseries}{}{.5em}{#1 \!\thmnote{\ #3}.}
\theoremstyle{maintheorem}
\newtheorem*{mainthm}{Theorem}

\makeatletter
\let\c@figure\c@thm
\let\c@table\c@thm
\makeatother
\crefname{figure}{Figure}{Figures}
\crefname{table}{Table}{Tables}

\newcommand{\case}[1]{\smallskip \noindent \textbf{Case #1:}}

\newcommand{\Aut}{\operatorname{Aut}}

\newcommand{\id}{\operatorname{id}}
\newcommand{\im}{\operatorname{im}}
\newcommand{\supp}{\operatorname{supp}}

\newcommand{\I}{\mathrm{I}}

\newcommand{\fab}[1]{{#1}^\mathrm{fab}}
\DeclareMathOperator{\Ore}{Ore}

\newcommand{\val}[2]{{#1}_{#2}}
\newcommand{\qval}[1]{\val{#1}{Q}}

\newcommand{\typeFP}[1]{\mathtt{FP}_{#1}}

\def\C{\mathbb{C}}
\def\R{\mathbb{R}}
\def\N{\mathbb{N}}
\def\Z{\mathbb{Z}}
\def\Q{\mathbb{Q}}

\def\1{\mathbbm{1}}

\def\K{\mathbb{K}}
\def\F{\mathbb{F}}

\newcommand{\D}{\mathcal{D}}

\def\s-{\smallsetminus}
\def\into{\hookrightarrow}

\def\iff{if and only if }

\newcommand{\nov}[3]{{\widehat{#1 #2}^{#3}}}
\newcommand{\novq}[2]{{\widehat{\Q #1}^{#2}}}

\newcommand{\oR}{{(-\infty,\infty]}}

\binoppenalty=\maxdimen
\relpenalty=\maxdimen

\newcounter{dawidcomments}

\author{Dawid Kielak}
\title[RFRS groups and virtual fibring]{Residually finite rationally solvable groups and virtual fibring}
\date{\today}

\begin{document}

\begin{abstract}
We show that a non-trivial finitely generated residually finite ra\-tio\-nal\-ly solv\-able (or RFRS) group  $G$ is virtually fibred, in the sense that it admits a virtual surjection to $\Z$ with a finitely generated kernel, \iff the first $L^2$-Betti number of $G$ vanishes. This generalises (and gives a new proof of) the analogous result of Ian Agol for fundamental groups of $3$-manifolds.
\end{abstract}

\maketitle
\section{Introduction}

In 2013, Ian Agol \cite{Agol2013} completed the proof of Thurston's Virtually Fibred Conjecture, which states that every closed hyperbolic $3$-manifold admits a finite covering which fibres over the circle. The final step consisted of showing that the fundamental group $G$ of a hyperbolic $3$-manifold has a finite index subgroup $H$ which embeds into a right-angled Artin group (a  RAAG). Agol in~\cite{Agol2008} showed that in this case $H$ is virtually \emph{residually finite rationally solvable} (or \emph{RFRS}); the same article contains the following result.

\begin{thm*}[Agol~{\cite{Agol2008}}]
Every compact irreducible orientable $3$-manifold with trivial Euler characteristic and non-trivial RFRS fundamental group admits a finite covering which fibres.
\end{thm*}
Since closed hyperbolic $3$-manifolds are irreducible and have vanishing Euler characteristic, it follows that such manifolds are virtually fibred, that is, admit fibring finite coverings.

Our focus here is the above result of Agol -- morally, we will generalise this theorem by removing the assumption of the group being the fundamental group of a $3$-manifold.

\begin{mainthm}[\ref{main statement}]
 Let $G$ be an infinite finitely generated group which is virtually RFRS. Then $G$ is virtually fibred, in the sense that it admits a finite-index subgroup mapping onto $\Z$ with a finitely generated kernel, \iff $\beta^{(2)}_1(G) = 0$.
\end{mainthm}

Here, $\beta^{(2)}_1(G)$ denotes the first $L^2$-Betti number. The assumption on the vanishing of $\beta^{(2)}_1(G)$ is supposed to be thought of as the analogue of Agol's insistence on the vanishing of the Euler characteristic of the $3$-manifold. Also, it is a  natural assumption here: the connection between fibring and the vanishing of the $L^2$-Betti numbers was exhibited by Wolfgang L\"uck (see \cref{Lueck}).

Let us remark  that virtually RFRS groups abound in nature: every subgroup of a RAAG is virtually RFRS (as shown by Agol~\cite{Agol2008}*{Corollary 2.3}), and so every virtually special group (in the sense of Wise) is virtually RFRS.

\smallskip

We show that \cref{main statement} implies Agol's result; in fact, we prove a slightly stronger, uniform version, see \cref{agol uniform} -- the same uniform version can in fact be deduced directly from Agol's theorem, see \cite{FriedlVidussi2015}*{Corollary 5.2}.

\subsection*{The method}
We start by observing that it follows from a result of Thomas Schick~\cite{Schick2002} that every RFRS group $G$ satisfies the (strong) Atiyah conjecture (over $\Q$), since it is residually \{torsion-free solvable\}. The group $G$ is also torsion-free, and so Linnell's theorem~\cite{Linnell1993} gives us a skew-field $\D(G)$ which contains $\Q G$. Moreover, the $\D(G)$-dimensions of $H_\ast(G;\D(G))$ coincide with the $L^2$-Betti numbers,
and so $\beta_1^{(2)}(G) = 0$ \iff $H_1(G;\D(G)) = 0$.
Now the main technical result of the current paper is \cref{DG is KG}, which says roughly  that $\D(G)$ is covered by the Novikov rings of finite index subgroups of $G$.


Which brings us to the Novikov rings, as defined by Sikorav~\cite{Sikorav1987}: For every homomorphism $\phi \colon G \to \R$ there exists the associated Novikov ring $\novq G \phi$, which contains the group ring $\Q G$. One should think of $\novq G \phi$ as a `completion' of $\Q G$ with respect to the `point at infinity' determined by $\phi$ -- more precisely, by the ray $\{ \lambda \phi \mid \lambda \in (0,\infty) \}$.
The key point for us here is the theorem of Sikorav~\cite{Sikorav1987} (originally proved for $\Z$-coefficients; we recast it in the setting of $\Q$-coefficients in \cref{Sikorav}), which states that a character $\phi \colon G \to \Z$ is fibred, that is has a finitely generated kernel, \iff $H_1(G;\novq G \phi) = H_1(G;\novq G {-\phi})= 0$.

At this point we have made a connection between the vanishing of the first $L^2$-Betti number of $G$, equivalent to
\[H_1(G;\D(G)) = 0\]
and fibring of a finite index subgroup $H$ of $G$, equivalent to
\[H_1(H;\novq H {\pm \psi}) = 0\]
and these two homologies are related by \cref{DG is KG}.

\smallskip
The one concept which we have not elucidated so far is that of RFRS groups. The definition is somewhat technical (see \cref{rfrs def}); the way RFRS groups should be thought of for the purpose  of this article is as follows: in a RFRS group $G$, every finite subset of $G$ can be separated in the abelianisation of some finite index subgroup.

Passing to a finite index subgroup is well-understood in terms of group homology; the significance of a finite subset $S$ of $G$ being separated in (the free part of) the abelianisation $\fab{G}$ lies in the following observation: if $x \in \Q G$ has support equal to $S$, then $x$ becomes invertible in $\novq G \phi$ for a `generic' $\phi \colon G \to \R$.

\smallskip
To summarise, here is a (rather rough) heuristic behind the proof of \cref{main statement}: We investigate a chain complex $C_\ast$ of free $\Q G$-modules computing the group homology of $G$. Since $G$ is a finitely generated group, the modules $C_1$ and $C_0$ are finitely generated modules. Since we only care about computing $H_1(G; ?)$, we modify $C_\ast$  by removing $C_i$ for $i \geqslant 3$ and replacing $C_2$ by a finitely generated module. Now, to compute  $H_1(G; \D(G))$ we can change the bases of $C_0, C_1$ and $C_2$ over $\D(G)$ and diagonalise the differentials. Crucially, the change of basis matrices are all finite, and hence by \cref{DG is KG} we obtain the same change of basis matrices in the Novikov ring $\novq H {\psi}$ of some finite index subgroup $H$ of $G$, for fairly generic $\psi$. This then implies the equivalence
\[
 H_1(H;\novq H {\psi}) = 0 \Leftrightarrow H_1(G;\D(G)) =0
\]
as desired.

We offer also a more general result, \cref{main thm chain cplxs},  showing the equivalence of vanishing of higher $L^2$-Betti numbers and higher Novikov homology groups $H_i(H;\novq H \phi)$.

\subsection*{The tools}

The main tools we use are the same as the ones used by the author in \cite{Kielak2018a}: twisted group rings, Linnell skew-field, Ore localisations, Malcev--Neumann completions and the Novikov homology. The flavour of the article is therefore rather ring-theoretic. All of the tools are discussed in detail in \cref{sec prelims}.

We have already discussed the Novikov homology; we also already met the skew-field of Linnell. 
The Ore localisation is a very simple tool allowing one to form fields of fractions from rings satisfying a suitable condition (known as the Ore condition). The elements of an Ore localisation are fractions of ring elements.

The Malcev--Neumann completion creates a skew-field containing $\Q G$ for groups $G$ which are biorderable -- the biordering allows us to consistently choose a way of building power series expansions for inverses of elements of $\Q G$, and thus we may use infinite series instead of fractions.

And, last but not least, twisted group rings: the slogan is that group rings are better behaved than groups with respect to short exact sequences. More specifically, start with a short exact sequence of groups
\[
 H \to G \to Q
\]
Choosing a section $s \colon Q \to G$ allows us to write every element $\sum \lambda_g g$ of $\Q G$ as \[\sum_{q \in Q} \Big(\sum_{h \in H} \lambda_{h s(q)} h\Big) s(q)\]
which clearly is an element of the group ring $(\Q H) Q$ via the identification $s(q) \mapsto q$. The caveat is that the ring multiplication in $(\Q H) Q$ is chosen in such a way that $(\Q H) Q$ is isomorphic as a ring to $\Q G$. The multiplication is given explicitly in \cref{sec tgr}; what is important here is that $Q$ acts on $\Q H$ by conjugation, and multiplication in $Q$ is given by the section $s$, and so there are correction terms appearing due to the fact that $s$ is not necessarily a homomorphism. The resulting ring is a \emph{twisted} group ring, which we will nevertheless denote by $(\Q H)Q$, since it behaves in essentially the same way as an honest group ring. Also, twisted group rings can be easily defined using other coefficients than $\Q$, which we will repeatedly use in the article.

\subsection*{Acknowledgements}
The author would like to thank: Peter Kropholler and Daniel Wise for helpful conversations; Stefan Friedl, Gerrit Herrmann, and Andrei Jaikin-Zapirain for comments on the previous version of the article; the referee for careful reading and the suggestion of both the statement and proof of \cref{thm cohdim 2}; the Mathematical Institute of the Polish Academy of Sciences (IMPAN) in Warsaw, where some of this work was conducted.

The author was supported by the grant KI 1853/3-1 within the Priority Programme 2026 \href{https://www.spp2026.de/}{`Geometry at Infinity'} of the German Science Foundation (DFG).

\section{Preliminaries}
\label{sec prelims}

\subsection{Generalities about rings}

Throughout the article, all rings are associative and unital, and the multiplicative unit $1$ is not allowed to coincide with the additive unit $0$ (this latter point has no bearing on the article; we state it out of principle). Ring homomorphisms preserve the units.
Also, we will use the natural convention of denoting the non-negative integers by $\N$. 

Before proceeding, let us establish some notation for ring elements.

\begin{dfn}
Let $R$ be a ring. A \emph{zero-divisor} is an element $x$ such that $xy= 0$ or $yx = 0$ for some non-zero $y \in R$.

An element $x$ is \emph{invertible} in $R$ \iff it admits both a right and a left inverse (which then necessarily coincide).
\end{dfn}

\begin{dfn}[Division closure]
Let $R$ be a subring of a ring $S$. We say that $R$ is \emph{division closed} in $S$ \iff every element in $R$ which is invertible in $S$ is already invertible in $R$.

The \emph{division closure} of $R$ is the smallest (with respect to inclusion) division-closed subring of $S$ containing $R$. It is immediate that the division closure of $R$ coincides with the intersection of all division closed subrings of $S$ containing $R$.
\end{dfn}


\subsection{Twisted group rings}
\label{sec tgr}

In this section, let $R$ be  a ring and let $G$ be a group.

\begin{dfn}[Support]
 Let $R^G$ denote the set of functions $G \to R$. Since $R$ has the structure of an abelian group, pointwise addition turns $R^G$ into an abelian group as well.

 Given an element $x \in R^G$, we will use both the function notation, that is we will write $x(g) \in R$ for $g \in G$, and the sum notation (standard in the context of group rings), that is we will write
 \[
  x = \sum_{g \in G} x(g) g
 \]
 When we use the sum notation, we do not insist on the sum being finite.

 Given $x \in R^G$, we define its \emph{support} to be
 \[
  \supp x = \{ g \in G \mid x(g) \neq 0 \}
 \]

 We define $RG$ to be the subgroup of $R^G$ consisting of all elements of finite support.
\end{dfn}

For every $H\leqslant G$, we identify $R^H$ with $R^G$ by declaring functions to be zero outside on $G \s- H$. In particular, this implies that $R H \subseteq R G$ (in fact, it is an abelian subgroup).

Throughout the paper, we will only make three identifications of sets; this was the first one, the second one is discussed in \cref{sec Ore}, and the third one in the proof of \cref{DH subfield}.

\smallskip
We will now endow $RG$ with a multiplication, turning it into a ring.

\begin{dfn}[Twisted group ring]
\label{twisted dfn}
Let $R^\times$ denote the group of units of $R$.
Let $\nu \colon G \to \Aut(R)$ and $\mu \colon G \times G \to R^\times$ be two functions satisfying
\begin{align*}
 \nu(g) \circ \nu(g') &= c\big(\mu(g,g')\big) \circ \nu(gg') \\
 \mu(g,g') \cdot \mu(gg',g'') &= \nu(g)\big(\mu(g',g'')\big) \cdot \mu(g,g'g'')
\end{align*}
where $c \colon R^\times \to \Aut(R)$ takes $r$ to the left conjugation by $r$. The functions $\nu$ and $\mu$ are the \emph{structure functions}, and turn $RG$ into a \emph{twisted group ring} by setting
\begin{equation}
\tag{$\star$}
\label{twisted conv}
r g \cdot r' g' = r \nu(g)(r') \mu(g,g') gg'
\end{equation}
and extending to sums by linearity (here we are using the sum notation).

When $\nu$ and $\mu$ are trivial, we say that $RG$ is \emph{untwisted} (in this case we obtain the usual group ring).
We adopt the convention that a group ring with $\Z$ or  $\Q$ coefficients is always untwisted.
\end{dfn}

We will suppress $\nu$ from the notation, and instead write $r^g$ for $\nu(g^{-1})(r)$.

\begin{rmk}
When $H \leqslant G$ is a subgroup, restricting the structure functions to, respectively, $H$ and $H \times H$ allows us to construct a twisted group ring $R H$. Observe that, as abelian groups, $RH \leqslant RG$ as $H$ is a subset of $G$, and thus the structure functions
turn $R H$ into a subring of $R G$.

It is also easy to see that when $H \lhdslant G$, then $R H$ is invariant under conjugation by elements of $G$.
\end{rmk}

\smallskip
The key source of twisted group rings is the following construction.

\begin{prop}
\label{twisted construction}
Let $Q$ be a quotient of $G$ with associated kernel $H$. Let $s \colon Q \to G$ be a set-theoretic section of the quotient map.
Let $R G$ be a twisted group ring.
 The structure functions $\nu \colon Q \to \Aut(R H)$ and $\mu \colon Q \times Q \to (R H)^\times$ given by, respectively,
 \[
  \nu(q) = c(s(q))
 \]
and
\[
 \mu(q,p) = s(q)s(p)s(qp)^{-1}
\]
satisfy the requirements of \cref{twisted dfn} and turn $(R H)Q$ into a twisted group ring.
\end{prop}
\begin{proof}
 We need to verify the two conditions of \cref{twisted dfn}. Firstly, for every $q,p \in Q$ we have
 \begin{align*}
 \nu(q) \circ \nu(p) &= c\big(s(q)\big) \circ c\big(s(p)\big) \\ &= c\big(s(q)s(p)\big) \\ &= c\big(s(q)s(p)s(qp)^{-1}s(qp)\big) \\ &= c\big(\mu(q,p)s(qp)\big) \\ &= c\big(\mu(q,p)\big) \circ \nu(qp)
 \end{align*}
Secondly, for every $p,q,r \in Q$ we have
\begin{align*}
 \mu(p,q) \cdot \mu(pq,r) &= s(p)s(q)s(pq)^{-1} \cdot s(pq)s(r)s(pqr)^{-1} \\
 &=s(p)s(q)s(r)s(pqr)^{-1} \\
  &=s(p)s(q)s(r)\cdot s(qr)^{-1}s(p)^{-1}\cdot s(p)s(qr)\cdot s(pqr)^{-1} \\
    &=s(p)s(q)s(r)s(qr)^{-1} s(p)^{-1}\cdot \mu(p,qr) \\
        &=c\big(s(p)\big)\big(s(q)s(r)\cdot s(qr)^{-1}\big) \cdot \mu(p,qr) \\
                &=\nu(p)\big(\mu(q,r) \big) \cdot \mu(p,qr) \qedhere
\end{align*}
\end{proof}

We will say that the twisted group ring $(R H) Q$ is \emph{induced} by the quotient map $G \to Q$. Note that the ring structure of $(R H) Q$ depends (a priori) on the section $s \colon Q \to G$.

\begin{rmk}
 Throughout, all sections are set-theoretic sections, and for every section $s$ we always require $s(1)=1$.
\end{rmk}

\begin{dfn}
 We extend the definition of $s \colon Q \to G$ to
 \[
  s \colon (R H) Q \to R G
 \]
by setting\[s( \sum_{q \in Q} x_q q) =\sum_{q \in Q} x_q s(q)\]
where $x_q \in R H$ for every $q \in Q$.
\end{dfn}

\begin{lem}
\label{QHQ is QG}
 The map $s$ is a ring isomorphism
 \[s \colon (R H) Q \overset{\cong}\to R G\]
\end{lem}
\begin{proof}
 The map is injective since $x_q s(q)$ and $x_p s(p)$ for $p \neq q$ are supported on distinct cosets of $H$ in $G$. The map is surjective since the support of every $x \in R G$ can easily be written as a disjoint union of finite subsets of these cosets.

 The fact that the map is a ring homomorphism follows directly from the definition of the twisted multiplication \eqref{twisted conv} in $(R H) Q$.
\end{proof}

The definition of the twisted group ring structure on $(R H) Q $ depends on the choice of $s$. But the above result shows that for two different sections, the twisted group rings $(R H) Q$ are isomorphic, and the isomorphism restricts to the identity on the coefficients $R H$.

\subsection{Ore condition}
\label{sec Ore}

\begin{dfn}
 Let $R$ be a ring, and let $T$ denote the set of its elements which are not zero-divisors. We say that $R$ satisfies the \emph{Ore condition} \iff for every $r \in R$ and $s \in T$ there exist $p,p' \in R$ and $q,q' \in T$ such that
 \[
  qr = ps \textrm{ and } rq' = s p'
 \]

We call $R$ an \emph{Ore domain} \iff $T = R \s- \{0\}$ and $R$ satisfies the Ore condition.
\end{dfn}

It is a classical fact that every ring satisfying the Ore condition embeds into its \emph{Ore localisation} $\mathrm{Ore}(R)$.
Elements of such an Ore localisation are expressions of the form $r/s$ (\emph{right fractions}) where $r \in R$ and $s \in T$; formally, they are equivalence classes of pairs $(r,s)$, with the equivalence relation given by multiplying both elements on the right by a non-zero-divisor.

Equivalently, the Ore localisation consists of \emph{left fractions} $s\backslash r$, and the Ore condition allows one to pass from one description to the other. This is needed in order to define multiplication in $\mathrm{Ore}(R)$.
A detailed discussion of the Ore localisation is given in Passman's book~\cite{Passman1985}*{Section 4.4}.

Since the map $r \mapsto r/1$ is an embedding, we will identify $R$ with its image in $\Ore(R)$. This is the second identification of sets, to which we alluded in \cref{sec tgr}.

When $R$ is an Ore domain, $\mathrm{Ore}(R)$ is a skew-field, sometimes called the
\emph{(classical) skew-field of fractions} of $R$. It is immediate that if $\F$ is a skew-field, then $\F = \Ore(\F)$, as every fraction $r/s$ is equal to $rs^{-1}$.

\smallskip
Let us state a useful (and immediate) fact about $\mathrm {Ore}(R)$.

\begin{prop}
\label{Ore facts}
Let $R$ and $S$ be rings satisfying the Ore condition.
Let
\[\rho \colon R \to S\]
be a ring homomorphism, and suppose that non-zero-divisors in $R$ are taken by $\rho$ to non-zero-divisors in $S$. Then
\begin{align*}
 \Ore(\rho) \colon \Ore(R) &\to \Ore(S) \\ r/s &\mapsto \rho(r)/\rho(s)
\end{align*}
is an injective ring homomorphism (which clearly extends $\rho$).
%
\end{prop}

We are interested in Ore localisations for the following reason.
\begin{thm}[Tamari~\cite{Tamari1957}]
\label{amen is Ore}
Let $RG$ be a twisted group ring of an amenable group $G$. If $RG$ is a domain, then it satisfies the Ore condition, and hence $\Ore(RG)$ is a skew-field.
\end{thm}

In fact, Tamari's theorem admits a partial converse: if $\Q G$ is an Ore domain then $G$ is amenable; see the appendix of \cite{Bartholdi2016}.

\subsection{Atiyah conjecture}

Let $G$ be a countable group.

\begin{dfn}
The \emph{von Neumann algebra} $\mathcal N(G)$ of $G$ is defined to be the algebra of bounded $G$-equivariant operators on $L^2(G)$.
\end{dfn}

Above, we view the operators as acting on the left, and the $G$ action is the action on the right.
Observe that a bounded (that is, continuous) $G$-equivariant operator $\zeta$ on $L^2(G)$ is uniquely determined by $\zeta(1) \in L^2(G)$. This allows us to view $\mathcal N (G)$ as a subset of $L^2(G)$, which in turn is a subset of $\C^G$. We will adopt this point of view, since it is convenient for us to work in $\C^G$.

The group ring $\Q G$ acts on $L^2(G)$ on the left by bounded operators, commuting with the right $G$-action. Thus, we have a map $\Q G \to \mathcal N(G)$. This map is easily seen to be injective. With the point of view adopted above we can in fact say much more: this map is the identity map $\Q G \leqslant \Q^G < \C^G$. Thus, we have $\Q G$ being an actual subset of $\mathcal N (G)$.

We need the von Neumann algebra only to define Linnell's ring $\D(G)$ (which we shall do in a moment), and so readers not familiar with von Neumann algebras should not feel discouraged.

\begin{thm}[\cite{Lueck2002}*{Theorem 8.22(1)}]
 The von Neumann algebra $\mathcal N(G)$ satisfies the Ore condition.
\end{thm}

Note that $\mathcal N(G)$ is extremely far from being a domain.

\begin{dfn}[Linnell ring]
\label{def D}
 Let $G$ be a group. We define $\D(G)$ (the \emph{Linnell ring}) to be the division closure of $\Q G$ inside $\Ore(\mathcal N(G))$.
\end{dfn}

It is important to note that left multiplication by $x$ for $x \in \Q G$ endows $\D(G)$ with a $\Q G$-module structure.

It is also important to observe that in L\"uck's book, $\D(G)$ denotes the division closure of $\C G$ rather than $\Q G$; this is a minor point, as all the desired properties are satisfied by our $\D(G)$ as well.


\begin{dfn}[Atiyah conjecture]
 We say that a torsion-free group $G$ satisfies the \emph{Atiyah conjecture} \iff $D(G)$ is a skew-field.
\end{dfn}

\begin{rmk}
 This is not the usual definition of the Atiyah conjecture; it is however equivalent to the (strong) Atiyah conjecture (over $\Q$) for torsion-free groups by a theorem of Linnell~\cite{Linnell1993}.
\end{rmk}

The Atiyah conjecture has been established for many classes of torsion-free groups. Crucially for us, it is known for residually \{torsion-free solvable\} groups by a result of Schick~\cite{Schick2002}. It is also known for virtually \{cocompact special\} groups \cite{Schreve2014} and locally indicable groups \cite{Jaikin-ZapirainLopez-Alvarez2018}, as well as for large classes of groups with strong inheritance properties.

\begin{lem}
\label{subfield of D}
Suppose that $G$ is a torsion-free group satisfying the Atiyah conjecture. Every sub-skew-field $\K$ of $\D(G)$ containing $\Q G$ is equal to $\D(G)$.
\end{lem}
\begin{proof}
Every non-zero element in $\K$ is invertible in $\K$, and therefore $\K$ is division closed in $\Ore(\mathcal N (G))$. Also, $\K$ contains $\Q G$. But $\D(G)$ is the intersection of all such rings, and hence $\D(G) = \K$.
\end{proof}

\smallskip

We will now discuss the behaviour of $\D(G)$ under passing to subgroups of $G$.

\begin{prop}[{\cite{Lueck2002}*{Lemma 10.4}}]
 \label{Atiyah subgrp}
 Every subgroup of a group satisfying the Atiyah conjecture also satisfies the Atiyah conjecture.
\end{prop}


\begin{prop}
\label{DH subfield}
Let $G$ be a group with a subgroup $H$.
The Linnell ring $\D(H)$ is a subring of $\D(G)$. Moreover, if $H$ is a normal subgroup then $\D(H)$ is invariant under conjugation by elements of $G$.
\end{prop}
\begin{proof}[Sketch proof]
Since the mathematical content of this proposition is standard (see for example \cite{Kielak2018a}*{Proposition 4.6} or \cite{Lueck2002}*{Section 10.2}), we offer only a sketch proof.

Take $\zeta \in \mathcal N(H)$. We will now make $\zeta$ act on $L^2(G)$. To this end, let $P$ be a set of coset representatives of $H$ in $G$. Note that $L^2(G)$ is equal to the set of $L^2$ functions in $\C^{\bigcup_{p\in P}Hp} = \prod_{p \in P} \C^{Hp}$, and hence is a subset of $\prod_{p \in P} L^2(Hp)$. We let $\zeta$ act on each $L^2(Hp)$ by $x \mapsto \zeta(xp^{-1})p$. It is immediate that the resulting action of $\zeta$ on $\prod_{p \in P} L^2(Hp)$ preserves $L^2(G)$; it is equally clear that the action is continuous and $G$-equivariant. This procedure produces a ring homomorphism $\mathcal N(H) \to \mathcal N(G)$, which we denote $\zeta \mapsto \zeta'$.

Recall that we are viewing $\mathcal N(H)$ and $\mathcal N(G)$ as subsets of $\C^H$ and $\C^G$, respectively. Since we have $\zeta(1) = \zeta'(1)$ by construction, the map $\zeta \mapsto \zeta'$ is actually a restriction of the identity map $\C^H \leqslant \C^G$.
Therefore, we have $\mathcal N(H) \leqslant \mathcal N(G)$.

Now an easy argument (using the decomposition of $G$ into cosets of $H$) shows that non-zero-divisors of $\mathcal N(H)$ are still non-zero-divisors in $\mathcal N(G)$. \cref{Ore facts} tells us that the identity map $\id$ induces an injective map
\[                                                                                                                                                                                                                                   \Ore(\id) \colon \Ore(\mathcal N(H)) \to \Ore(\mathcal N(H))                                                                                                                                                                                                                                             \]
Formally speaking, since Ore localisation consist of equivalence classes, we are not allowed to say that $\Ore(\mathcal N(H))$ is  a subset of $ \Ore(\mathcal N(H))$. We never view elements of Ore localisations as equivalence classes however -- we view them always as fractions $r /s$, and such expressions in $\Ore(\mathcal N(H))$ are also valid expressions in $\Ore(\mathcal N(G))$. We will therefore identify (for the third and final time) the set $\Ore(\mathcal N(H))$ with the corresponding subset of $\Ore(\mathcal N(G))$.
It is now easy to see that we also have $\D(H) \leqslant \D(G)$.

\smallskip
Now suppose that $H$ is a normal subgroup, and let $g \in G$. Conjugation gives us an action of $g$ on $\C^G$, and this action clearly preserves $\C^H$. It is immediate that it also preserves $\mathcal N(H)$ and $\Q H$, and hence has to preserve $\D(H)$.
\end{proof}


\smallskip

We will now concentrate on the situation in which $H$ is a normal subgroup of $G$.
\begin{rmk}
\label{DHQ def}
 The above proposition immediately implies that  if $H$ is the kernel of an epimorphism $G \to Q$, we may form a twisted group ring $\D(H) Q$, where the structural functions are precisely as in \cref{twisted construction}. We will say that $\D(H) Q$ is \emph{induced by $G \to Q$}.
\end{rmk}

In practice, we will be interested in two cases: $Q$ will be either finite or free-abelian. Let us first focus on the case of $Q$ being finite.

\begin{lem}
\label{DHQ is DG}
\label{DHQn is DG}
If $Q$ is finite, and $G$ is torsion free and satisfies the Atiyah conjecture, then the natural map $\D(H) Q \to \D(G)$ given by\
\[
\sum_{q \in Q} \lambda_q q \mapsto \sum_{q \in Q} \lambda_q s(q)
\]
 is a ring isomorphism.
\end{lem}

The statement above can be proven completely analogously to \cite{Lueck2002}*{Lemma 10.59}.  Note that L\"uck is only assuming $\D(G)$ to be semisimple, whereas we are working under a stronger assumption of $\D(G)$ (and hence also $\D(H)$) being skew-fields.

As before, we will refer to the isomorphism $\D(H) Q \to \D(G)$ above simply as $s$.

\smallskip
Now we move to the case of $Q$ being free abelian.


\begin{prop}
\label{DG is Ore}
Suppose that $G$ is finitely generated, torsion free, and satisfies the Atiyah conjecture. Let $s_\alpha$ denote a section of an epimorphism $\alpha \colon G \to Q$ onto a free-abelian group; let $K=\ker \alpha$. The twisted group ring $\D(K) Q$ satisfies the Ore condition, and the map
\[s_\alpha \colon \D(K) Q \to \D(G)\]
is injective and induces an isomorphism
\[\mathrm{Ore}(s_\alpha) \colon \mathrm{Ore}\big(\D(K) Q \big) \cong \D(G)\]
\end{prop}
\begin{proof}[Sketch proof]
The proof is completely analogous to that of \cite{Lueck2002}*{Lemma 10.69}, and hence we only offer an outline here.

First we note that, as before, our $\D(G)$ is not the same as L\"uck's, but it exhibits the same properties. L\"uck starts with a short exact sequence with $G$ in the middle, and the quotient being virtually finitely generated abelian. We are looking at a short exact sequence
\[
 K \to G \to Q
\]
where the quotient is finitely generated and free abelian, and so certainly satisfies the assumptions of \cite{Lueck2002}*{Lemma 10.69}. The lemma also requires $\D(K)$ to be semi-simple, but in our case we have the stronger assumption that $\D(K)$ is a skew-field, since $K$ is torsion-free and satisfies the Atiyah conjecture.
L\"uck concludes that the Ore localisation of $\D(K) Q$ exists and is isomorphic to $\D(G)$; the isomorphism produced coincides with our $\Ore(s_\alpha)$.
\end{proof}

\subsection{\texorpdfstring{$L^2$}{L\texttwosuperior}-homology}

\begin{dfn}[$L^2$-Betti numbers]
 Let $G$ be a torsion-free group satisfying the Atiyah conjecture.
Let $C_\ast$ be a chain complex of free left $\Q G$-modules.
 The \emph{$n^{th}$ $L^2$-Betti number}   of $C_\ast$, denoted $\beta_n^{(2)}(C_\ast)$, is defined to be the dimension of $H_n(\D(G) \otimes_{\Q G} C_\ast )$ viewed as a (left) $\D(G)$-module. Note that there is no ambiguity here, since dimensions of modules over skew-fields (which are essentially vector spaces) are well-defined.

 We say that $C_\ast$ is $L^2$-acyclic \iff $\beta_n^{(2)}(C_\ast) = 0$ for every $n$.

 Similarly, we define the \emph{$n^{th}$ $L^2$-Betti number}   of $G$, denoted $\beta_n^{(2)}(G)$, to be the dimension of $H_n(G;\D(G))$ viewed as a $\D(G)$-module; we say that $G$ is $L^2$-acyclic \iff $\beta_n^{(2)}(G) = 0$ for every $n$.
\end{dfn}
\begin{rmk}
 Again, this is not the usual definition of $L^2$-Betti numbers. In the above setting it coincides with the usual definition by \cite{Lueck2002}*{Lemma 10.28(3)}.
\end{rmk}

The following is a version of a celebrated theorem of L\"uck~\cite{Lueck2002}*{Theorem 1.39} (it is stated here in a slightly reworded fashion).

\begin{thm}[{\cite{Lueck2002}*{Theorem 7.2(5)}}]
 Let
 \[
  1 \to H \to G \to K \to 1
 \]
be an  exact sequence of groups such that $\beta_p^{(2)}(H)$ is finite for all $p \leqslant d$. Suppose that $K$ is infinite amenable or suppose that $BK$ has finite $d$-skeleton and there is an injective endomorphism $K \to K$ whose image is proper and of finite index in $K$. Then $\beta_p^{(2)}(G) = 0$ for every $p \leqslant d$.
\end{thm}

Restricting the above to the situation we will be interested in gives the following.

\begin{thm}
\label{Lueck}
If $H$ is a finitely generated group, and if
 \[
  1 \to H \to G \to \Z \to 1
 \]
is an  exact sequence of groups, then $\beta_1^{(2)}(G) = 0$.
\end{thm}
\begin{proof}
We need only observe that $\Z$ is an infinite amenable group, that being finitely generated implies finiteness of the zeroth and first $L^2$-Betti numbers.
\end{proof}

\subsection{Biorderable groups}

\begin{dfn}
 A \emph{biordering} of a group $G$ is a total order $\leqslant$ on $G$ such that for every $a,b,c \in G$ we have
 \[
  a \leqslant b \Rightarrow ac \leqslant bc \textrm{ and } a \leqslant b \Rightarrow ca \leqslant cb
 \]

 A group $G$ is \emph{biorderable} \iff it admits a biordering.
\end{dfn}

\begin{dfn}[Malcev--Neumann skew-field]
 Let $\leqslant$ denote a biordering of a group $G$. Let $R$ be a ring, and let $RG$ be a twisted group ring. The set
 \[
 \F_\leqslant (RG) =  \{ x \colon G \to R \mid \supp x \textrm{ is $\leqslant$-well ordered}  \}
 \]
becomes a skew-field when endowed with the twisted convolution \eqref{twisted conv} -- see \cites{Malcev1948,Neumann1949}. We call it the \emph{Malcev--Neumann skew-field} of $RG$ with respect to $\leqslant$. (Recall that a totally ordered set is \emph{well-ordered} \iff every non-empty subset thereof admits a unique minimum.)
\end{dfn}

\section{Novikov rings}

\subsection{Novikov rings}

\begin{dfn}[Novikov rings]
Let $RG$ be a twisted group ring, and let
\[\phi \in H^1(G;\R) = \mathrm {Hom}(G;\R)\]
(we will say that $\phi$ is a \emph{character}).
We set
\[
 \nov R G \phi = \Big\{ x \colon G \to R \, \Big\vert \, \big\vert \supp x \cap \phi^{-1}((\infty, \kappa]) \big\vert < \infty \textrm{ for every } \kappa \in \R
 \Big\}
\]
The set $\nov R G \phi$ is turned into a ring using the twisted convolution formula \eqref{twisted conv}. We will call it the \emph{Novikov ring} of $RG$ with respect to $\phi$.
\end{dfn}

\begin{dfn}
 Let $\oR = \R \sqcup \{\infty\}$ be endowed with the obvious total ordering and an abelian monoid structure given by addition (with the obvious convention that $x+ \infty = \infty$ for every $x\in \oR$).

Let $\phi \in  H^1(G;\R)$ be given.
We extend the definition of $\phi \colon G \to \R$ to
\[\phi \colon \nov R G \phi \to \oR\]
by setting
\[
 \phi(x) = \left\{ \begin{array}{ccl} \min \phi \big(\supp(x)\big) & \textrm{if} & x \neq 0\\
                                        \infty & \textrm{if} & x = 0 \end{array} \right.
\]
In particular, this also defines  $\phi \colon R G \leqslant \nov R G \phi \to \oR$.
\end{dfn}

\begin{rmk}
\label{novikov series}
 Let $(x_i)_i$ be a sequence of elements of $\nov R G \phi$ with $(\phi(x_i))_i$ converging to $\infty$. It is easy to see that the partial sums
 \[
  \sum_{i=0}^N x_i
 \]
converge pointwise as functions $G \to R$, and hence one can easily define their limit $\sum_{i=0}^\infty x_i \in \nov R G \phi$.
\end{rmk}

\begin{lem}
\label{properties of nivikov}
 Let $R G$ be a twisted group ring, and let $\phi \in H^1(G;\R)$ be a character. For every $x,y \in \nov R G \phi$ we have
 \begin{enumerate}
 \item $\phi(x + y) \geqslant \min \{ \phi(x), \phi(y) \}$;
 \label{pon item 1}
  \item $\phi(xy) \geqslant \phi(x) + \phi(y)$;
  \label{pon item 2}
  \item $\phi(xy) = \phi(x) + \phi(y)$ if $RG$ has no zero-divisors.
  \label{pon item 3}
 \end{enumerate}
\end{lem}
\begin{proof}
\begin{enumerate}
 \item This follows immediately from the definition of $\phi \colon \nov R G \phi \to \oR$ and the fact that $\supp (x+y) \subseteq \supp (x) \cup \supp(y)$.
 \item Let $a \in \supp (x)$ and $b \in \supp (y)$ be such that
 \[
 \phi(x) = \phi(a) \textrm{ and } \phi(y) = \phi(b)
 \]
 For every $a' \in \supp (x)$ and $b' \in \supp (y)$ we have
 \[
 \phi(a') \geqslant \phi(a) \textrm{ and } \phi(b') \geqslant \phi(b)
 \]
 Now, every $c \in \supp (xy)$ can be written as a product $c = a'b'$ of some $a'$ and $b'$ as above. Thus
 \[
 \phi(xy) \geqslant \min_{a',b'} \phi(a'b') = \min_{a',b'} \big( \phi(a') + \phi(b') \big) = \phi(a) + \phi(b) = \phi(x) + \phi(y)
 \]

\item If $R G$ has no zero-divisors, we know that $ab \in \supp (xy)$, and so
\[
 \phi(xy) = \phi(a) + \phi(b) =  \phi(x) + \phi(y) \qedhere
\]
\end{enumerate}
\end{proof}

\begin{rmk}
\label{phi is a morph}
 It follows immediately from (3) above, that if $RG$ has no zero-divisors, then $\phi \colon R G \s- \{0\} \to \R$ is a monoid homomorphism.
\end{rmk}

\begin{dfn}[Free abelianisation]
Let $G$ be a group. Let $\fab G$ denote the image of $G$  in $H_1(G;\Q)$ (under the natural homomorphism).
When $G$ is finitely generated, $\fab G$ is simply the \emph{free part of the abelianisation} of $G$, that is abelianisation of $G$ divided by the torsion part.
We define $\alpha \colon G \to \fab G$ to be the natural epimorphism. We will refer to $\alpha$ as the \emph{free abelianisation map}.
\end{dfn}

\begin{rmk}
 Clearly, every character in $H^1(G;\R)$ yields a character $\fab{G} \to \R$ in the obvious way. Conversely, every character in $H^1(\fab{G};\R)$ yields a character in $H^1(G;\R)$. Thus we will not differentiate between such characters, and use the same symbol to denote both.
\end{rmk}

\begin{prop}
\label{novikov standard}
 Let $G$ be a group, and let $\phi \in H^1(G;\R)$. Choose a section $s_\alpha \colon \fab G \to G$, and recall that this choice gives $(\Q K) \fab G$ the structure of a twisted group ring, where $K = \ker \alpha$. The ring isomorphism $s_\alpha \colon (\Q K) \fab G \to \Q G$ extends to a ring isomorphism
 \[ \nov {(\Q K)} {\fab G} \phi \cong \novq G \phi\]
\end{prop}
\begin{proof}
%
We define $\iota \colon {\Q K}^{\fab{G}} \to \Q^G$ by setting $\iota(x)(g) = x(\alpha(g))(g s_\alpha(\alpha(g))^{-1})$ for every $x \colon \fab{G} \to \Q K$.
Let us unravel this definition:
we take $\alpha(g)$, an element of the the group $\fab{G}$, and evaluate $x$ on it. This yields an element of $\Q K$, which is formally a subset of $\Q^{K}$, the set of functions from $K$ to $\Q$. Thus, it makes sense to evaluate $x(\alpha(g))$ on elements in $K = \ker \alpha$, and $g s_\alpha(\alpha(g))^{-1}$ is such an element. After this last evaluation we obtain an element of $\Q$, as desired.

When restricted to $(\Q K) \fab{G}$, the map $\iota$ coincides with the map $s_\alpha$ (which is an isomorphism by \cref{QHQ is QG}): this is easy to verify, but cumbersome to write down, so we omit it here.


We claim that $\iota(\nov {(\Q K)} {\fab{G}} {\phi})$ lies within $\novq G \phi$. Indeed, take $x \in \nov {(\Q K)} {\fab{G}} {\phi}$. By definition, for every $\kappa \in \R$ the set $\supp x \cap \phi^{-1}((-\infty,\kappa]) \subseteq \fab{G}$  is finite.

Observe that for $a \in \fab{G}$ we have
\[
 \supp \iota(x) \cap \alpha^{-1}(a) = \{g \in G \mid \alpha(g) = a \textrm{ and } x(a)(g s_\alpha(\alpha(g))^{-1}) \neq 0 \}
\]
Since for every $a \in \fab{G}$ the function $x(a) \in \Q K$ has finite support in $K$,
we have
\[
 \vert \supp \iota(x) \cap \alpha^{-1}(a) \vert < \infty
\]
for every $a \in \fab{G}$.
Therefore $\supp \iota(x) \cap \phi^{-1}((\infty,\kappa]) \subseteq G$ is finite, as it is  a finite union of finite sets.
This shows that $\iota(x) \in \novq G \phi$, as claimed.

There is an obvious inverse to $\iota$ defined on $\novq G \phi$, namely $x \in \novq G \phi$ being mapped to $x' \colon \fab{G} \to \Q K$ given by $x'(a)(k) = x(ks(a))$.  It is immediate that $x' \in \nov {(\Q K)} {\fab{G}} \phi$. We now verify that this map is an inverse to $\iota$:
\[
 \iota(x')(g) = x'\big(\alpha(g)\big)\big(g s_\alpha(\alpha(g))^{-1}\big) = x\big(g s_\alpha(\alpha(g))^{-1} s_\alpha(\alpha(g))\big) = x(g)
\]
and
\[
 \iota(x)'(a)(k) = \iota(x)(ks_\alpha(a)) = x( \alpha(k s_\alpha(a))\big(ks_\alpha(a) s_\alpha(\alpha(ks_\alpha(a)))^{-1}\big) = x(a)(k) \qedhere
\]

\end{proof}

Again, we will denote the isomorphism $\nov {(\Q K)} {\fab{G}} \phi \overset{\cong}\to \novq G \phi$ by $s_\alpha$.

\subsection{BNS invariants and Sikorav's theorem}

Let $G$ be a finitely generated group, and let $X$ denote the Cayley graph of $G$ with respect to some finite generating set $T$. Recall that $G$ is the vertex set of $X$, and each edge comes with an orientation and a label from $T$.

\begin{dfn}
 A  character $\phi \in H^1(G;\R) \s- \{0\}$ lies in the \emph{Bieri--Neumann--Strebel} (or \emph{BNS}) \emph{invariant} $\Sigma(G)$ \iff the full subgraph of $X$ spanned by $\{ g \in G \mid \phi(g) \geqslant 0 \}$ is connected.
\end{dfn}

The invariant $\Sigma(G)$ does not depend on the choice of a finite generating set -- this can easily be proven directly, but it will also follows from Sikorav's theorem below.

The significance of BNS invariants for us lies in the following result.

\begin{thm}[Bieri--Neumann--Strebel~{\cite{Bierietal1987}*{Theorem B1}}]
\label{BNS fibring}
Let $G$ be a finite\-ly generated group, and let $\phi \colon G \to \Z$ be non-trivial. Then $\ker \phi$ is finitely generated \iff $\{\phi, -\phi\} \subseteq \Sigma(G)$.
\end{thm}

In \cite{Sikorav1987} Sikorav gave a way of computing $\Sigma(G)$ using the Novikov rings. He used the rings $\nov \Z G \phi$, whereas we have to work with $\novq G \phi$. For this reason we will prove Sikorav's theorem in our setting (the proof is identical).

\begin{thm}[Sikorav's theorem]
\label{Sikorav}
Let $G$ be a finitely generated group, and let $\phi \in H^1(G;\R) \s- \{0\}$. The following are equivalent:
\begin{enumerate}
 \item $\phi \in \Sigma(G)$;
 \item $H_1(G;\nov \Z G \phi) = 0$;
 \item $H_1(G;\novq  G \phi) = 0$.
\end{enumerate}
\end{thm}
\begin{proof}
Let us start with some general setup: Since $\phi \neq 0$, there exists a generator $s \in T$ with $\phi(s) \neq 0$. Without loss of generality we may assume that $\phi(s)>0$.

Let
\[C_1 \to C_0\]
denote the cellular chain complex of $X$ with coefficients in $\Z$. It is immediate that $C_1$ and $C_0$ are finitely generated free $\Z G$-modules. In fact, $C_0$ is $1$-dimensional, and we can append the chain complex on the right with the augmentation map $\Z G \to \Z$. We obtain
\[
 C_1 \to C_0 \to \Z
\]
and it is immediate that this complex is exact at $C_0$. We append the complex again, this time on the left by adding a free $\Z G$-module $C_2$ in such a way that
\[
C_2 \to  C_1 \to C_0 \to \Z
\]
is exact at $C_1$ and $C_0$. We denote the above chain complex by $C_\ast$. The differentials will be denoted by $\partial$.

We pick cellular bases for the free modules $C_0$ and $C_1$. For $C_0$ the basis we pick is the singleton $\{1\}$ -- recall that group elements are the vertices of $X$. For $C_1$ the basis is the collection of oriented edges emanating from $1$. We will denote the basis element corresponding to an edge labelled $t \in T$ by $e_t$.

Setting the value of $\phi$ on the basis elements of $C_1$ to zero gives as a map
$\phi \colon C_1 \to \oR$.
Specifically, we have $\phi(g.e_t) = \phi(g)$, which assigns a value to every edge in $X$.

It is clear that $\{ e_t \mid t \in T \}$ forms a basis of $ \nov \Z G \phi \otimes_{\Z G} C_1$, and we may view chains in  $\nov \Z G \phi \otimes_{\Z G} C_1$ as functions from the edge set $E$ of $X$ to $\Z$.

\smallskip

\noindent
$\mathbf{(1) \Rightarrow (2)}$
In this part of the  proof we adopt the convention that tensoring without specified ring is always over $\Z G$.

Since $\phi(s) \neq 0$, the element $1-s$ is invertible in $\nov \Z G \phi$ with inverse $\sum_{i\in \N} s^i$. Therefore the space of $1$-cycles in $\nov \Z G \phi \otimes C_\ast$ is spanned by
\[
 e_t' = e_t - (1-t)(1-s)^{-1}e_s
\]
with $t \in T \s- \{s\}$. In fact it is easy to see that these elements form a basis of the space of $1$-cycles.
We will now define a $\nov \Z G \phi$-linear map $c$ from the $1$-cycles to $\nov \Z G \phi \otimes C_2$ by specifying its value on every basis element  $e_t'$.

Take $t \in T \s- \{s\}$. Since $\phi(s) > 0$, there exists $n \in \N$ such that $\phi(ts^n)>\phi(s)$ and $\phi(s^n) >\phi(s)$.
We are assuming that $\phi \in \Sigma(G)$, and the definition tells us that we may connect $s^{-1} t s^n$ to $s^{-1} s^n $ in $X$ by a path which goes only through vertices $g$ with $\phi(g) \geqslant 0$.
This implies that the path traverses only edges whose $\phi$ value is non-negative.

We now act (on the left) by $s$, and conclude that the existence of a
 $1$-chain $p$ (supported on the image  of the path we just constructed) with boundary $\partial p = t s^n - s^n$, and such that $\phi$ evaluated at the edges in the support of $p$ is always positive. Since $C_\ast$ is exact at $C_1$, there exists $c_t \in C_2$ such that
 \begin{equation}
  \partial c_t = e_t + t \sum_{i=0}^{n-1} s^i e_s -p - \sum_{i=0}^{n-1} s^i e_s  \tag{$\ddagger \ddagger$}
 \label{sikorav eq}
 \end{equation}
since the right-hand side is a cycle -- it is supported on a cycle in $X$ composed of the edge corresponding to $e_t$, the path underlying $p$, and two segments, one connecting $1$ to $s^n$ and the other connecting $t$ to $ts^n$. We define $c(e'_t) = c_t$.

Now $\partial \circ c (e_t') = \partial c_t $, which we have already computed. Since $\partial c_t$ is a cycle, we can write it using the basis $\{e'_{t} \mid t \in T \s- \{s\} \}$ -- this is done by forgetting $e_s$ in \eqref{sikorav eq}, and replacing $e_t$ by $e'_t$. We obtain
\[\partial c_t = e_t' - \sum_{r \in T \s- \{s\}}\lambda_r e_r'\]
where $\lambda_r \in \nov \Z G \phi$ is such that $\phi(\lambda_r)>0$ for every $r$ -- this follows from the properties of $p$. Writing $\partial \circ c$ as a matrix, with respect to the basis $\{e'_{t} \mid t \in T \s- \{s\} \}$, we obtain
\[\I - M \]
where $\phi$ takes every entry of $M$ to $(0,\infty)$. This matrix is invertible over $\nov \Z G \phi$ with inverse $\sum_{i \in \N} M^i$ (compare \cref{novikov series}), and therefore $\partial \circ c$ is an epimorphism. Hence the differential
\[\partial \colon \nov \Z G \phi \otimes C_2 \to \nov \Z G \phi \otimes C_1 \]
is onto the cycles, and thus $H_1(\nov \Z G \phi \otimes C_\ast) = 0$, as claimed.

\smallskip
\noindent
$\mathbf{(2) \Rightarrow (3)}$ This is immediate since $\novq G \phi =  \novq G \phi \otimes_{\nov \Z G \phi} \nov \Z G \phi $, and tensoring is right-exact.

\smallskip
\noindent
$\mathbf{(3) \Rightarrow (1)}$
We start by replacing $C_\ast$ by $\Q \otimes_\Z C_\ast$; in order not to make the notation too cumbersome, we will continue to use $C_\ast$. Also, in this part, unspecified tensor products are taken over $\Q G$.

We are assuming that
\[
 0 = H_1(G;\novq G \phi) = H_1(\novq G \phi \otimes C_\ast)
\]
Take $g,h \in G$ with $\phi(g), \phi(h) \geqslant 0$. We need to show that there exists a path in $X$ going only through vertices $x \in G$ with $\phi(x) \geqslant 0$ connecting $g$ to $h$.

Since $X$ is connected, there exists a path $p$ in $X$ connecting $g$ to $h$. We take $p$ to be geodesic. Recall that  $\phi(s) > 0$.  Consider the function $\xi \colon E \to \Q$ which is the characteristic function of the set of edges of $p$ (taking orientations into account) and the edges of the infinite rays $r_1$ and $r_2$ emanating from $g$ and $h$ and using only edges with label $s$ (and only positively  orientated).

The function $\xi$ clearly lies in $\novq G \phi \otimes C_1$; by construction, it is actually a cycle there. Since $H_1(\novq G \phi \otimes C_\ast) = 0$, there exists an element $x \in \novq G \phi \otimes C_2$ with $\partial x = \xi$.

Since $C_2$ and $\novq G \phi \otimes C_2$ are free modules, there exists a finite-dimensional free submodule $C_2' \leqslant C_2$ such that
\[
 x \in  \novq G \phi \otimes C_2'
\]
The module $C_2'$ has a finite basis; the boundary of each of the basis elements is a chain in $C_1$. The chosen basis for $C_2'$ gives also a basis for  $\novq G \phi \otimes C_2'$, and so we may write $x = (x_1, \dots, x_m)$ for some $m \in \N$ with $x_i \in \novq G \phi$. By taking the restriction of each $x_i \colon G \to \Q$ to suitable finite subset of $G$ we form elements $x_i' \in \Q G$ such that
\[
 \phi\big(\partial (x_i - x_i')\big)>0
\]
for every $i$. Let $x' = (x_1', \dots, x_m') \in C_2' \leqslant C_2$. By construction, the support of $\partial (x-x')$ contains
only edges whose value under $\phi$ is positive. Now
\[
 \xi = \partial x = \partial(x-x') + \partial x'
\]
and so $\xi - \partial x'$ is a cycle in $\novq G \phi \otimes C_1$ supported only on edges of positive $\phi$-value.

Note that $\partial(x')$ has finite support, and the rays $r_1$ and $r_2$ are both infinite. Thus, $\xi - \partial x'$ contains in its support some edges of $r_1$ and $r_2$. Since $\xi - \partial x'$ is a cycle, there exists a path in $\supp (\xi - \partial x')$, say $q$, connecting some endpoints, say $x_1$ and $x_2$, of these edges. Now a path starting at $g$, following $r_1$ up to the vertex $x_1$ (without loss of generality), then following $q$, and then following the ray $r_2$ against its orientation to $h$, is as desired.
\end{proof}

\subsection{Linnell's skew-field and Novikov rings}

Suppose that we have a finitely generated group $G$ which is torsion free and satisfies the Atiyah conjecture. We have already discussed that, in this case, $\Q G$ embeds into Linnell's skew-field $\D(G)$.
We pick a character $\phi \colon G \to \R$. We also pick a section  $s_\alpha$ of the free abelianisation map $\alpha$, and let $K = \ker \alpha$. Recall that we will also treat $\phi$ as a character $\phi\colon \fab{G} \to \R$.

The key point of this section is to find a skew-field (depending on a character $\phi$) that will simultaneously house $\novq G \phi$ and $\D(G)$. We will use it to define a set of good elements of $\D(G)$, which we will call \emph{representable in $\novq G \phi$}, by taking the intersection of $\D(G)$ and $\novq G \phi$ in the larger object.
To achieve this goal, we need to study a multitude of injective maps between objects we have previously constructed. These are summarised in \cref{maps diagram}, which is a commutative diagram.

\begin{figure}
\[
 \xymatrix{
\Q G \ar@{^{(}->}[dd] \ar@{^{(}->}[rrr]
&
&
& \D(G) \ar@{^{(}->}@/^5pc/[dd]^{i_\leqslant}
\\
& (\Q K) {\fab{G}} \ar@{^{(}->}[r] \ar@{_{(}->}[ul]^{s_\alpha}_{\cong} \ar@{^{(}->}[d]
& \D(K) {\fab{G}} \ar@{^{(}->}[r] \ar@{^{(}->}[ur]^{s_\alpha} \ar@{^{(}->}[d] \ar@{^{(}->}[dr]_{\id_\leqslant}
& \mathrm{Ore}(\D(K) {\fab{G}}) \ar@{_{(}->}[u]_{\mathrm{Ore}(s_\alpha)}^\cong \ar@{^{(}->}[d]^{\mathrm{Ore}(\id_\leqslant)}
\\
\novq G \phi \ar@{_{(}->}@/_2pc/[rrr]_{j_\phi}
& \nov {(\Q K)} {\fab{G}} \phi \ar@{_{(}->}[l]^{s_\alpha}_{\cong} \ar@{^{(}->}[r]
& \nov {\D(K)} {\fab{G}} \phi \ar@{^{(}->}[r]_{\id_\leqslant}
& \F_\leqslant(G)
 }
\]
\caption{A commutative diagram of the various embeddings}
\label{maps diagram}
\end{figure}
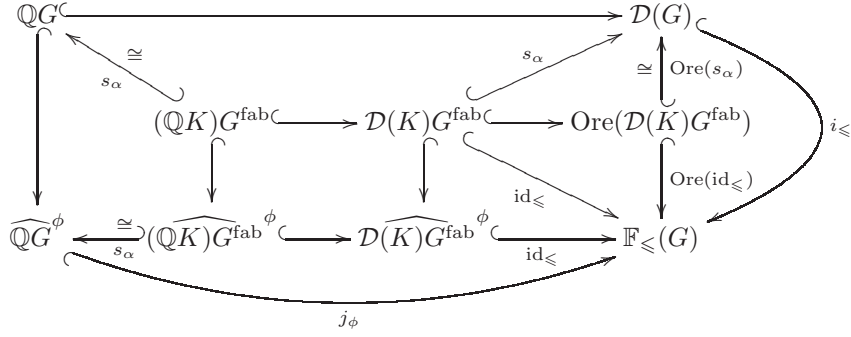

In the diagram, we see three maps labelled $s_\alpha$. These are all induced by the set-theoretic section $s_\alpha \colon \fab{G} \to G$ in a way we had made explicit previously.

All the unlabelled maps are various identity maps.

We pick a biordering $\leqslant$ on the finitely generated free-abelian group $\fab{G}$ which is \emph{compatible} with $\phi$, in the sense that $\phi \colon \fab{G} \to \R$ is an order-preserving homomorphism. Let $\F_\leqslant(G)$ denote the Malcev--Neumann skew-field $\F_\leqslant( \D(K) \fab{G} )$ associated to $\leqslant$. By definition, $\F_\leqslant(G)$ is a subset of $\D(K)^{\fab{G}}$. The ring $\nov {\D(K)} {\fab{G}} \phi$ is another such subset, and in fact we have
\[
 \nov {\D(K)} {\fab{G}} \phi \subseteq \F_\leqslant(G)
\]
since $\phi$ is order preserving; we will denote the inclusion map by $\id_\leqslant$. This map is located in the bottom-right part of \cref{maps diagram}. Note that $\id_\leqslant$ does not in fact depend on $\leqslant$. We are adding the subscript for clarity.
We define
\[
j_\phi = \id_\leqslant \circ \id \circ {s_\alpha}^{-1} \colon \novq G \phi \to \F_\leqslant(G)
\]

Composing
\[
\id_\leqslant \circ \id \colon \D(K) \fab{G} \to \F_\leqslant(G)
\]
gives us an inclusion which we will also call $\id_\leqslant$; again, $\id_\leqslant$ does not depend on $\leqslant$.
Now $\Ore(\id_\leqslant) \colon \Ore(\D(K) \fab{G}) \to \F_\leqslant(G)$ denotes the map induced by $\id_\leqslant \colon \D(K) \fab{G} \to \F_\leqslant(G)$, as in \cref{Ore facts}; this last map is injective, as it is a non-trivial ring homomorphism from a skew-field.
We define
\[
 i_\leqslant = \Ore(\id_\leqslant) \circ \Ore(s_\alpha)^{-1} \colon \D(G) \to \F_\leqslant(G)
\]
Note that $\Ore(\id_\leqslant)$ and $i_\leqslant$ are the only maps in \cref{maps diagram} which depend on the choice of $\leqslant$.

This completes the description of the embeddings visible in \cref{maps diagram}.

\begin{dfn}[Representable elements]
 We say that an element $x \in \D(G)$ is \emph{representable} in $\novq G \phi$ \iff we have $i_\leqslant(x) \in  j_\phi(\novq G \phi)$ for every biordering $\leqslant$ compatible with $\phi$.
 %
\end{dfn}

\begin{lem}
\label{iotaC}
Let $x\in \D(G)$ be representable in $\novq G \phi$. If $\leqslant$ and $\leqslant'$ are biorderings on ${\fab{G}}$ compatible with $\phi$, then
\[
 i_\leqslant(x) = i_{\leqslant'}(x)
\]
\end{lem}
\begin{proof}
Recall that $x = s_\alpha(p) s_\alpha(q)^{-1}$ with $p,q \in \D(K) {\fab{G}}$ and $q \neq 0$.

Now,
\begin{align*}
 i_\leqslant(x) &= \Ore(\id_\leqslant)\Ore(s_\alpha)^{-1}(s_\alpha(p) s_\alpha(q)^{-1}) \\
 &= \Ore(\id_\leqslant)(p/q) \\
  &= \id_\leqslant(p)\id_\leqslant(q)^{-1}
\end{align*}
and similarly for $\leqslant'$.
Since $x$ is representable, both $i_\leqslant(x)$ and $i_{\leqslant'}(x)$ are elements of the ring $\nov {\D(K)} {\fab{G}} \phi$ (which is a subset of $\D(K)^{\fab{G}}$).
But we also know that \[\id_\leqslant(p) = \id_{\leqslant'}(p) = \id(p) \in \nov {\D(K)} {\fab{G}} \phi\]
and
\[\id_\leqslant(q) = \id_{\leqslant'}(q) = \id(q) \in \nov {\D(K)} {\fab{G}} \phi\]
and thus
\begin{align*}
\big( i_\leqslant(x) - i_{\leqslant'}(x)\big) \id(q) &=
\id_\leqslant(p)\id_\leqslant(q)^{-1}\id_\leqslant(q) - \id_{\leqslant'}(p)\id_{\leqslant'}(q)^{-1} \id_{\leqslant'}(q) \\
&= \id_\leqslant(p) - \id_{\leqslant'}(p)\\
&= 0
\end{align*}
But the ring $\nov {\D(K)} {\fab{G}} \phi$ embeds in a skew-field (via $\id_\leqslant$), and hence it does not admit zero-divisors. Since $q \ne 0$, we conclude that
\[
 i_\leqslant(x) = i_{\leqslant'}(x) \qedhere
\]
\end{proof}

In view of the above, we will denote by $\iota_\phi(x) \in \novq G \phi$ the unique element such that for any (hence every) suitable $\leqslant$ we have
\[
 j_\phi\big(\iota_\phi(x) \big) = i_\leqslant(x)
\]
We think of $\iota_\phi$ as a map taking representable elements in $D(G)$ to their representatives in $\novq G \phi$. It is easy to see that representable elements form a subring of $\D(G)$, and $\iota_\phi$ is a ring homomorphism.

\begin{dfn}[$\D(G,S)$]
For $S \subseteq H^1(G;\R)$ we define $\D(G,S)$ to be the subset of $\D(G)$ consisting of elements representable over $\novq G \phi$ for every $\phi \in S$.

When $S = \{ \phi \}$, we will write $\D(G,\phi)$ for $\D(G,\{\phi\})$.
\end{dfn}


\begin{rmk}
\label{DHU is a ring}
It is clear that $\D(G,S)$ is a subring of $\D(G)$, as it is an intersection of subrings.
\end{rmk}

\begin{lem}
 \label{making DHU Qequiv}
 Let $S, S' \subseteq H^1(G;\R)$ be two subsets with $S' \subseteq S$. Then
 \[
  \D(G,S) \leqslant \D(G,S')
 \]
 is a subring.
\end{lem}
\begin{proof}
 This follows immediately from the definitions.
\end{proof}


\smallskip

Since we are working with Novikov rings, it is essential for us to be able to evaluate $\phi$ at elements in the various rings appearing in \cref{maps diagram}. On all rings but the ones in the right-most column of \cref{maps diagram} we already know how to evaluate $\phi$, since all these rings are group rings or Novikov rings of $G$ or $\fab{G}$. We will now focus on the right-most column.

\begin{dfn}
 Let $\phi\colon G \to \R$ be a character, and let $p,q \in \D(K) \fab{G}$ be given, with $q \neq 0$. We have already defined $\phi(p) \in \oR$ and $\phi(q) \in \R$. We now set
 \[
\phi(p/q) = \phi(p) - \phi(q)
 \]
It is immediate that $\phi(pr/qr) = \phi(p/q)$ for $r \neq 0$, and hence this formula defines the evaluation $\phi \colon \Ore{(\D(K) \fab{G})} \to \oR$.

Since $\Ore(s_\alpha)$ is an isomorphism, we define $\phi \colon \D(G) \to \oR$ by
\[
 \phi(s_\alpha(p)s_\alpha(q)^{-1}) = \phi(p/q) = \phi(p) - \phi(q)
\]

Finally, let $x \in \F_\leqslant(G)$. If $x = 0$ we set $\phi(x) = \infty$ as usual. Now suppose that $x \neq 0$. By definition, $x \in \D(K)^{\fab G}$, and $\supp x$ is well-ordered. In particular, $\supp x$ has a unique $\leqslant$-minimum, say $a$. We set $\phi(x) = \phi(a)$.
\end{dfn}

We can now add all the different evaluation maps called $\phi$ to \cref{maps diagram}. We claim that in doing so, we produce another commutative diagram -- in other words, the various evaluations defined above coincide.

The claim is trivial to verify everywhere, except for the maps
\[
\id_\leqslant \colon \nov {\D(K)} {\fab G} \phi \to \F_\leqslant (G)
\]
and
\[
\Ore(\id_\leqslant) \colon \Ore( {\D(K)} {\fab G}) \to \F_\leqslant (G)
\]

For the former map, take $x \in \nov {\D(K)} {\fab G} \phi$; we may assume that $x \neq 0$. By definition, $\phi(x) = \phi(a)$ where $a \in \supp x$ is an element minimising $\phi$. But, since $\leqslant$ is $\phi$-compatible, we may take $a$ to be the $\leqslant$-minimum of $\supp x$.

For the latter map, take $p,q \in \D(K) \fab{G} \s- \{0\}$. Let
\[a = \min_\leqslant \supp p\]
and
\[b = \min_\leqslant \supp q\]
Arguing as above, we see that
\[
 \phi(p/q) = \phi(p) - \phi(q) = \phi(a) - \phi(b)
\]
On the other hand, the construction of $\F_\leqslant(G)$ tells us that
\[\min_\leqslant \supp \id_\leqslant(q)^{-1} = b^{-1}\]
and so
\[
 \phi\big(\Ore(\id_\leqslant)(p/q)\big) = \phi\big( \min_\leqslant \supp (\id_\leqslant p \id_\leqslant(q)^{-1}) \big) = \phi(ab^{-1}) = \phi(a) - \phi(b)
\]
and we are done.


\begin{rmk}
\label{phi homom}
It follows immediately from the Ore condition and \cref{phi is a morph} (noting that $\D(K) \fab{G}$ has no zero-divisors, as it embeds in a skew-field $\F_\leqslant(G)$) that $\phi$ is now a well-defined group homomorphism $\D(G) \s- \{0\} \to \R$.

It is also immediate that for $x \in \D(G)$ the map $H^1(G;\R) \to \R$ given by $\phi \mapsto \phi(x)$ is continuous.
\end{rmk}

\subsection{Groups with finite quotients}

Let $G$ be a finitely generated  torsion-free group satisfying the Atiyah conjecture (so $\D(G)$ is a skew-field). Let   $Q$ be a finite quotient of $G$ with the corresponding kernel $H$. We fix a section $s \colon Q \to G$ of the quotient map. As always, we have $s(1) = 1$.

We will also use $\alpha_H \colon H \to \fab{H}$ to denote the free abelianisation map of $H$, and we pick a section $s_{\alpha_H}$ of this map.

\begin{dfn}
For every $\psi \in H^1(H;\R)$ we define $\psi^q \colon H \to \R$ to be the character given by $\psi^q(x) = \psi(s(q) x s(q)^{-1})$.
\end{dfn}
Note that the notation $\psi^q$ is reasonable, since
\begin{align*}
 (\psi^q)^p(x) &= \psi^q\big(s(p) x s(p)^{-1}\big) \\
 &= \psi\big(s(q) s(p) x s(p)^{-1} s(q)^{-1}\big) \\
 &= \psi(\mu(q,p)s(qp) x s(qp)^{-1}\mu(q,p)^{-1}) \\
 &= \psi^{qp}( x )
\end{align*}
since $\mu(q,p) \in H$ and therefore conjugating the argument by $\mu(q,p)$ does not alter the value under $\psi$.

\smallskip
Note that conjugation by elements $s(q)$ with $q \in Q$ induces an outer action of $Q$ on $H$. But an outer action descends to an action of $Q$ on $H^1(H;\R)$.

Let $S \subseteq H^1(H;\R)$ be a subset invariant under the action of $Q$. Recall that we have a twisted group ring $\D(H) Q$ induced by the quotient $G \to Q$. Since $S$ is $Q$-invariant, it is immediate that the action of $Q$ on $\D(H)$ preserves $\D(H,S)$. Indeed, take $x \in \D(H,S)$ and $\psi \in S$. By definition, we have $\iota_\psi(x) \in \novq H \psi$. In particular, $\iota_\psi(x)$ is an element of $\Q^H$.  Take $q \in Q$. The element $s(q)^{-1} \iota_\psi(x) s(q)$ is clearly still an element of $\Q^H$, since $H$ is normal in $G$. It is immediate that in fact  $s(q)^{-1} \iota_\psi(x) s(q) \in \novq H {\psi^q}$, and $\psi^q \in S$ since $S$ is $Q$-invariant.

Now that we know that $\D(H,S)$ is $Q$-invariant, the structure functions of $\D(H) Q$ give us a twisted group ring structure for $\D(H,S) Q$.

Interestingly, if $S$ is not necessary $Q$-invariant, the abelian group $\D(H,S)Q$ whose elements are formal sums of elements in $Q$ with coefficients in $\D(H,S)$, and the addition is pointwise, forms a right $\D(H,S')Q$-module, where $S' = \bigcap_{q\in Q} S^q$. The right action of  $\D(H,S')Q$ on $\D(H,S)Q$ is induced by the following rule
\[
 x_q q \cdot y_p p = x_q {y_p}^{q^{-1}} \mu(q,p) qp
\]
where $q,p \in Q$, $x_p \in \D(H,S)$ and $y_p \in \D(H,S')$, and where we treat
\[{y_p}^{q^{-1}} \in \D(H,S')\]
as an element in $\D(H,S)$, and the multiplication $x_p {y_p}^{q^{-1}}$ is carried out in $\D(H,S)$.

Similarly, we endow the abelian group $\F_\leqslant(H) Q$, consisting of formal sums of elements in $Q$ with coefficients in $\F_\leqslant(H)$, with the structure of a right $\D(H)Q$ module:
the right action is induced by almost the same rule as above, that is,
\[
 x_q q \cdot y_p p = x_q i_\leqslant ( {y_p}^{q^{-1}} ) \mu(q,p) qp
\]
where now $x_q$ is an element of $\F_\leqslant(H)$, and ${y_p}^{q^{-1}}\in \D(H)$ is sent to an element of $\F_\leqslant(H)$ by the map $i_\leqslant = \Ore(\id_\leqslant) \Ore(s_{\alpha_H})^{-1}$ (compare \cref{maps diagram}).

\begin{dfn}[Value and defect]
Let $\psi \in H^1(H;\R)$ be a character. For $x \in \D(H) Q$ we write $x = \sum_{q \in Q} x_q q$ with $x_q \in \D(H)$.
We define the \emph{$Q$-value} of $\psi$ at $x$ to be
\[
 \qval{\psi}(x) = \min_{p,q \in Q} \big\{ \psi^p(x_{q}) + \frac 1 {\vert Q \vert} \psi(s(q)^{\vert Q \vert}) \big\} \in \oR
\]

For a character $\psi \in H^1(H;\R)$ we define its \emph{$Q$-defect}
\[
\vert \psi \vert_Q = \max_{p,q\in Q} \Big\vert \qval{\psi}\big( s(p)s(q)s(pq)^{-1} \big) - \qval{\psi}(p)-\qval{\psi}(q) + \qval{\psi}(pq) \Big\vert \in [0,\infty)
\]
(note that $s(p)s(q)s(pq)^{-1} \in H \subseteq \D(H)$).
\end{dfn}

 In order not to overload notation, we will not differentiate between characters in $H^1(G;\R)$ and their restriction to $H$ in  $H^1(H;\R)$.

\begin{lem}
\label{qval for phi}
 For every $\phi \in H^1(G;\R)$, $q \in Q$, and $x \in \D(H)Q$, we have $\phi^q = \phi$, $\vert \phi \vert_Q = 0$, and $\qval{\phi}(x) = \phi(s(x))$.
\end{lem}
\begin{proof}
Let $x = \sum_{q\in Q} x_q q$ with  $x_q \in \D(H)$. Note that $x_1$ is an arbitrary element of $\D(H)$. If $x_1 = 0$ then
\[
 \phi^q(x_1) = \phi(s(q) x_1 s(q)^{-1}) = \phi(0) = \phi(x_1)
\]
If $x_1 \neq 0$ then
again
\[
 \phi^q(x_1) = \phi(s(q) x_1 s(q)^{-1}) = \phi(x_1)
\]
since $\phi$ a homomorphism from $\D(G) \s- \{0\}$ to an abelian group.

We also have
\[
 \qval{\phi}(x) = \min_{p,q \in Q} \big\{ \phi^p(x_{q}) + \frac 1 {\vert Q \vert} \phi(s(q)^{\vert Q \vert}) \big\} = \min_{q} \phi(x_q s(q)) = \phi(s(x))
\]

Finally, we have
\[
\vert \phi \vert_Q = \max_{p,q\in Q} \Big\vert \qval{\phi}\big( s(p)s(q)s(pq)^{-1} \big) - \qval{\phi}(p)-\qval{\phi}(q) + \qval{\phi}(pq) \Big\vert = \max_{p,q\in Q} \left\vert \phi(1)  \right\vert = 0
\]
which completes the proof
\end{proof}

Let us now investigate the key properties of the $Q$-value map \[\qval{\psi}  \colon \D(H) Q \to \oR\]

\begin{lem}
\label{ineq}
Let $\psi \in H^1(H;\R)$, $x,y \in \D(H)$, $q \in Q$, $w,z \in \D(H) Q$ be given.
All of the following hold:
\begin{enumerate}
\item $\qval{\psi}(x^{q^{-1}}) = \qval{\psi}(x)$;
\label{ineq 5}
 \item $\qval{\psi}(xq) = \qval{\psi}(q \cdot x) = \qval{\psi}(q) + \qval{\psi}(x)$;
  \label{ineq 1}
 \item $\qval{\psi}(xy) \geqslant \qval{\psi}(x) + \qval{\psi}(y)$;
 \label{ineq 2}
 \item $\qval{\psi}(x+y) \geqslant \min \{ \qval{\psi}(x), \qval{\psi}(y) \}$;
 \label{ineq 6}
 \item $\qval{\psi}(z+w)\geqslant \min \{ \qval{\psi}(z), \qval{\psi}(w) \}$;
  \label{ineq 3}
 \item $\qval{\psi}(zw) \geqslant \qval{\psi}(z) + \qval{\psi}(w) - \vert \psi\vert_Q$.
  \label{ineq 4}
\end{enumerate}
\end{lem}
\begin{proof}
\begin{enumerate}
\item We have
\begin{align*}
\qval{\psi}(x^{q^{-1}}) &= \qval{\psi} \big(s(q) x s(q)^{-1} \big) \\
&= \min_{p \in Q} \psi^{pq^{-1}} \big(s(q) x s(q)^{-1}\big) \\
&= \min_{p} \psi\big(s(pq^{-1}) s(q) x s(q)^{-1}s(pq^{-1})^{-1}\big) \\
&= \min_{p} \psi\big( s(pq^{-1}) s(q) s(p)^{-1} s(p) x s(p)^{-1} s(p) s(q)^{-1} s(pq^{-1})^{-1}\big) \\
 &\overset{(\dagger)}= \min_{p} \psi\big(  s(p) x s(p)^{-1} \big) \\
&= \min_p \psi^p(x) \\
&= \qval{\psi}(x)
\end{align*}
 where $(\dagger)$ follows from the fact that $s(pq^{-1}) s(q) s(p) \in H \subseteq \D(H) \s- \{0\}$ and that $\psi$ restricted to $\D(H) \s- \{0\}$ is a homomorphism with abelian image, and so is stable on conjugacy classes.

 \item By definition,
 \[\qval{\psi}(xq)=  \min_p \big\{ \psi^p(x)\big\} + \frac 1 {\vert Q\vert} \psi\big(s(q)^{\vert Q\vert} \big) = \qval{\psi}(x) + \qval{\psi}(q) \]
 Also,
\[\qval{\psi}(q \cdot x)= \qval{\psi}(x^{q^{-1}} q) =   \qval{\psi}(x^{q^{-1}}) + \qval{\psi}(q) \overset{\eqref{ineq 5}}{=} \qval{\psi}(x) +\qval{\psi}(q) \]

 \item 
 We have
 \begin{align*}
  \qval{\psi}(xy)
  &= \min_p \psi^p(x y) \\
   &\overset{(\ddagger)}= \min_p \big( \psi^p(x)+ \psi^p(y) \big) \\
   &\geqslant \min_p \psi^p(x) + \min_p \psi^p(y) \\
   &= \qval{\psi}(x) + \qval{\psi}(y)
 \end{align*}
 where $(\ddagger)$ follows from \cref{phi homom}.

\item 
The Ore condition tells us that we may write $x = s_\phi(x') s_\phi(d)^{-1}$ and $y = s_\phi(y') s_\phi(d)^{-1}$ for $x',y',d \in {\D(K)} \fab{G}$. Thus we have
\begin{align*}
\qval{\psi}(x+y) &= \min_p \psi^p(x + y) \\
&= \min_p \{ \psi^p(x' + y') - \psi^p(d) \} \\
&\overset{(\dagger \dagger)}{\geqslant}  \min_p \min \{ \psi^p(x') - \psi^p(d), \psi^p(y') - \psi^p(d) \} \\
&=  \min_p \min \{ \psi^p(x), \psi^p(y) \} \\
  &=  \min\{ \psi^p(x), \psi^p(y) \mid  p \in Q\} \\
 &= \min \{\min_p \psi^p(x), \min_p \psi^p(y) \} \\
 &= \min \{\qval{\psi}(x), \qval{\psi}(y) \}
\end{align*}
where $(\dagger \dagger)$ follows from \cref{properties of nivikov}\eqref{pon item 1}.

\item We write $z = \sum_p z_p p$ and $w = \sum_p w_p p$ with $z_p,w_p \in \D(H)$. Now, $\qval{\psi}(z+w)$ is defined as the minimum over $p\in Q$ of
\begin{align*}
\qval{\psi}\big((z_p +w_p)p\big)
&\overset{\eqref{ineq 1}}{=} \qval{\psi}(z_p +w_p) + \qval{\psi}(p) \\
&\overset{\eqref{ineq 6}}\geqslant \min \{\qval{\psi}(z_p) , \qval{\psi}(w_p) \} + \qval{\psi}(p) \\
&\overset{\eqref{ineq 1}}{=}  \min \{\qval{\psi}(z_p p) , \qval{\psi}(w_p p)\} \\
&\overset{\phantom{(4)}}{\geqslant} \min \{\qval{\psi}(z) , \qval{\psi}(w)\}
\end{align*}

\item Again, we write $z = \sum_p z_p p$ and $w = \sum_p w_p p$. Now, for every $p,p'\in Q$ we have
\begin{align*}
\qval{\psi}(z_p p \cdot w_{p'} p') &= \qval{\psi}(z_p {w_{p'}}^{ p^{-1}} s(p)s(p')s(pp')^{-1} pp')\\
 &\overset{\eqref{ineq 1}}{=} \qval{\psi}(z_p {w_{p'}}^{ p^{-1}} s(p)s(p')s(pp')^{-1}) + \qval{\psi}( pp')
\end{align*}
 Observe that $s(p)s(p')s(pp')^{-1} \in H$ and hence we may continue:
\begin{align*}
\qval{\psi}(z_p p \cdot w_{p'} p')
  &\overset{\eqref{ineq 2}}{\geqslant}  \qval{\psi}(z_p)+ \qval{\psi} ({w_{p'}}^{p^{-1}})+ \qval{\psi}\big(s(p)s(p')s(pp')^{-1} \big) + \qval{\psi}(pp')\\
   &\geqslant  \qval{\psi}(z_p)+ \qval{\psi} ({w_{p'}}^{p^{-1}}) - \vert \psi \vert_Q + \qval{\psi}(p) +\qval{\psi}(p')\\
    &\overset{\eqref{ineq 5}}{=}  \qval{\psi}(z_p)+ \qval{\psi} (w_{p'})- \vert \psi \vert_Q + \qval{\psi}(p) +\qval{\psi}(p') \\
    &\overset{\eqref{ineq 1}}=  \qval{\psi}(z_p p)+ \qval{\psi} (w_{p'} p')- \vert \psi \vert_Q
\end{align*}
There exists $p''\in Q$ such that
\[\qval{\psi}(zw)= \qval{\psi}\big(\sum_{pp'=p''} z_p p \cdot w_{p'} p'\big)\]
 since $\qval{\psi}(zw)$
 is defined as a minimum of such terms. Now we have
 \begin{align*}
 \qval{\psi}\Big(\sum_{pp'=p''} z_p p \cdot w_{p'} p'\Big) &\overset{\eqref{ineq 3}}{\geqslant} \min_{pp'=p''} \qval{\psi} ( z_p p \cdot w_{p'} p') \\
 &\geqslant \min_{pp'=p''}\big( \qval{\psi}(z_p p)+ \qval \psi (w_{p'} p')- \vert \psi \vert_Q \big)\\
 &= \min_{pp'=p''}\big( \qval{\psi}(z_p p)  + \qval \psi (w_{p'}p') \big) - \vert \psi \vert_Q \\
  &\geqslant \min_{p} \qval{\psi}(z_p p) + \min_{p} \qval \psi (w_{p} p)  - \vert \psi \vert_Q \\
    &= \qval{\psi}(z) + \qval{\psi} (w)  - \vert \psi \vert_Q  \qedhere
 \end{align*}
\end{enumerate}
\end{proof}

We are now ready for the key lemma which will allow us to construct inverses for elements in $\D(H,S)Q$.

\begin{dfn}
For every $\psi \in H^1(H;\R)$ we set $\psi^Q = \{ \psi^q \mid q \in Q\}$.
\end{dfn}

\begin{lem}
 \label{key lem}

 Let $\psi \in H^1(H;\R)$.
 Let $x,y \in \D(H,\psi^Q)Q$ be given. Suppose that $x$ is invertible in $\D(H,\psi^Q)Q$ with inverse $x^{-1}$. Suppose further that we have
\[
\qval{\psi}(y) + \qval{\psi}(x^{-1}) - 2\vert \psi \vert_Q >0
\]
Then the inverse of $s(x+y)$ in $\D(G)$ exists and lies in $s(\D(H,\psi)Q)$.
\end{lem}
\begin{proof}
Note that  the map $s \colon \D(H) Q \to \D(G)$ is an isomorphism by \cref{DHQ is DG}.

 We start by remarking that $x + y\neq 0$, as otherwise we would have (using \cref{ineq}(\ref{ineq 4}))
 \[\qval{\psi}(y) + \qval{\psi}(x^{-1}) = \qval{\psi}(-x) + \qval{\psi}(x^{-1}) \leqslant \qval{\psi}(-1) + \vert \psi \vert_Q = \vert \psi \vert_Q\]
which would contradict the assumed inequality.
Therefore $s(x+y)$ is invertible in $\D(G)$, with inverse $s(z)$ where $z = \sum_{q\in Q} z_q q$ with $z_q \in \D(H)$. Our goal is to show that in fact every $z_q$ is representable in $\novq H \psi$.
\smallskip

We have
\[
 \qval{\psi}(x^{-1}y) \geqslant \qval{\psi}(y) + \qval{\psi}(x^{-1}) - \vert \psi \vert_Q > \vert \psi \vert_Q
\]
by \cref{ineq}\eqref{ineq 4} and by assumption.
Let us quantify this inequality: there exists $\epsilon>0$ such that
\[
 \qval{\psi}(x^{-1}y) \geqslant \vert \psi \vert_Q + \epsilon
\]
Now repeated application of \cref{ineq}\eqref{ineq 4} tells us that for every $i > 0$ we have
\[
\qval{\psi}\big( (-x^{-1}y)^i \big) = \qval{\psi}\big( (x^{-1}y)^i \big) \geqslant \vert \psi \vert_Q + i\epsilon
\]

As $(-x^{-1}y)^i \in \D(H) Q$, we may write $(-x^{-1}y)^i = \sum_{q \in Q} w_{i,q} q$ with $w_{i,q} \in \D(H)$. Now
\[
 i \epsilon \leqslant \qval{\psi}\big( (-x^{-1}y)^i\big) = \min_{p,q \in Q} \{ \psi^p(w_{i,q}) + \frac 1 {\vert Q\vert} \psi(s(q)^{\vert Q \vert}) \} \leqslant \psi(w_{i,q}) + \frac 1 {\vert Q\vert} \psi(s(q)^{\vert Q \vert})
\]
for every $q \in Q$.
Since $Q$ is finite and $\psi$ is fixed, this implies that the sequence $(\psi(w_{i,q}))_i$ converges to $\infty$ for every $q \in Q$.
Since $x^{-1}y \in \D(H,\psi^Q)Q$, we have
\[\iota_\psi(w_{i,q}) \in \novq H \psi\]
which together with the previous observation yields the convergence of the series
\[
 \sum_{q\in Q}  \left( \sum_{i=0}^\infty  \iota_\psi(w_{i,q}) \right) q
\]
in $\novq H \psi Q$ (compare \cref{novikov series}); we denote the limit by $z'$.
(We will show later that $z'$ is the element representing $zx$ in $\novq H \psi Q$.)

\smallskip
Let us pick a biordering $\leqslant$ on $\fab{H}$ compatible with $\psi$. By definition, we have
\[j_\psi(z') \in \F_\leqslant(H) Q \]
(this is a slight abuse of notation, as $j_\psi$ is really defined as a map $\novq H \psi \to \F_\leqslant(H)$; here we extend it by the identity on $Q$; we will similarly abuse notation for $i_\leqslant$).

We have already observed that $\F_\leqslant(H) Q$ is a right $\D(H) Q$-module.
We claim that in the module $\F_\leqslant(H) Q$ we have
\[j_\psi(z') \cdot (1 + x^{-1}y) = 1\]
Indeed, this can be seen by taking the partial sums
\[
 j_\psi\left( \sum_{q\in Q}  \left( \sum_{i=0}^m  \iota_\psi(w_{i,q}) \right) q   \right) = \sum_{q\in Q}  \left( \sum_{i=0}^m  i_\leqslant(w_{i,q}) \right) q =  i_\leqslant\left( \sum_{i=0}^m  (-x^{-1}y)^i\right)
\]
and multiplying them by $1 + x^{-1}y$ on the right. Observing that the module action happens via the map $i_\leqslant$, we obtain
\begin{align*}
 i_\leqslant\left( \sum_{i=0}^m (-x^{-1}y)^i \right) \cdot (1 + x^{-1}y) &= i_\leqslant\left( \sum_{i=0}^m (-x^{-1}y)^i (1 + x^{-1}y)  \right) \\
 &= i_\leqslant \big( 1- (-x^{-1}y)^{m+1} \big) \\
 &= 1 - i_\leqslant \big( (-x^{-1}y)^{m+1} \big) \\
  &= 1 - i_\leqslant\left( \sum_{q\in Q} w_{m+1,q} q \right)
\end{align*}
We have already shown that $\psi\left( \sum_{q\in Q} w_{m+1,q} q \right)$ tends to $\infty$ with $m$, and hence
\[
 \lim_{m \longrightarrow \infty} \iota_\psi \left( \sum_{q\in Q} w_{m+1,q} q \right) = 0
\]
Therefore we also have
\[
 \lim_{m \longrightarrow \infty} j_\psi \iota_\psi \left( \sum_{q\in Q} w_{m+1,q} q \right) = 0
\]
But $j_\psi \iota_\psi = i_\leqslant$ and so
\begin{align*}
 j_\psi(z') \cdot (1 + x^{-1}y) &= \lim_{m \longrightarrow \infty}  i_\leqslant\left( \sum_{i=0}^m \left(-x^{-1}y\right)^i \right) \cdot (1 + x^{-1}y) \\
 &= 1 - \lim_{m \longrightarrow \infty} i_\leqslant\left( \sum_{q\in Q} w_{m+1,q} q \right) \\
 &= 1
\end{align*}

as claimed.

\smallskip

Recall that we have $z \in \D(H)Q$ which is an inverse of $x+y$ in $\D(H)Q$.
 We have also
\[
 {i_\leqslant}(z x) \cdot (1 + x^{-1}y) = {i_\leqslant}\big(z(x+y)\big) = 1
\]
We conclude that $\big(j_\psi(z')-{i_\leqslant}(z x)\big) \cdot (1 + x^{-1}y) = 0$. But then
\[
 j_\psi(z')-{i_\leqslant}(z x) = \big(j_\psi(z')-{i_\leqslant}(z x)\big) \cdot \big( (1 + x^{-1}y) \cdot zx \big)
= 0
\]
and so ${i_\leqslant}(zx) \in \im(j_\psi)$. But this is the precise meaning of $zx$ being representable in $\novq H \psi Q$. Since $x^{-1} \in \D(H,\psi^Q)Q$, we conclude that $z$ is representable in $\novq H \psi Q$. 
\end{proof}

\section{RFRS}

\subsection{Generalities}

\begin{dfn}[RFRS]
\label{rfrs def}
A group $G$ is called \emph{residually finite rationally solvable} or \emph{RFRS}  \iff there exists a sequence $(H_i)_{i \in \N}$ of finite index normal subgroups $H_i \lhdslant G$ with $H_0=G$ such that
\begin{itemize}
\item $H_{i+1} \leqslant H_i$ (that is, $(H_i)_i$ is a \emph{chain}), and
 \item $\bigcap H_i = \{1\}$ (that is, $(H_i)_i$ is a \emph{residual chain}), and
 \item $\ker \alpha_i \leqslant H_{i+1}$ for every $i$, where $\alpha_i \colon H_i \to \fab{H_i}$ is the free abelianisation map.
\end{itemize}
The sequence $(H_i)_i$ is called a \emph{witnessing chain}, and it always comes equipped with set-theoretic sections $s_i \colon Q_i = G/H_i \to G$ of the quotient maps $\beta_i \colon G \to Q_i$ such that $s_i(1) = 1$.
\end{dfn}

\begin{prop}
\label{rfrs atiyah}
Every RFRS group $G$ is torsion-free and satisfies the Atiyah conjecture.
\end{prop}
\begin{proof}
We will show that $G$ is residually \{torsion-free solvable\}. This will suffice, as residually torsion-free groups are themselves torsion free, and residually \{torsion-free solvable\} groups satisfy the Atiyah conjecture by \cite{Schick2002}*{Proposition 1 and Theorem 1}.

Let us now proceed with the proof.
Set $K_i = \ker \alpha_i \lhdslant H_i$.
Since $K_i \leqslant H_i$, it is immediate that $\bigcap K_i = \{1\}$. We have $K_i \lhdslant H_{i+1}$ by the definition of RFRS.
Further, we have $K_{i+1} \leqslant K_i$
since $K_{i+1}$ is the intersection of the kernels of all homomorphisms $H_{i+1} \to \Q$, and every homomorphism $H_i \to \Q$ restricts to a homomorphism $H_{i+1} \to \Q$.
The group $K_{i}/K_{i+1}$ is torsion-free abelian since $H_{i+1}/K_{i+1}$ is. Therefore $(K_i)_{i\in \N}$ is a residual chain witnessing the fact that $G$ is residually \{torsion-free solvable\}.
\end{proof}

\subsection{Linnell's skew-field of finitely generated RFRS groups}

Let $G$ be a finitely generated RFRS group with witnessing chain $(H_i)_{i\in \N}$ and finite quotients $\beta_i \colon G \to Q_i$ with kernels $H_i$ and sections $s_i \colon Q_i \to G$.

Note that $H^1(H_i;\R) = \mathrm{Hom} (H_i,\R)$ is naturally a subspace of $H^1(H_j;\R)$ for every $j \geqslant i$.


\begin{dfn}[Rich and very rich subsets]
 A character $\psi \in H^1(H_n;\R)$ is \emph{irrational} \iff it is injective on $\fab{H_n}$. Note that characters irrational in $H^1(H_n;\R)$ are not necessarily irrational in $H^1(H_{n+1};\R)$.

Let $\mathcal{L} = \{L_{-m} ,  \dots , L_{-1} \}$ denote a finite chain of linear subspaces of $H^1(G;\R)$ with $L_{-m} < \dots < L_{-1}$. We endow each $L_i$ with a fixed dense subset; we will refer to the characters in this subset as \emph{irrational}. The dense subsets of irrational characters of the subspaces $L_i$ are part of the structure of $\mathcal L$.

Suppose that $\mathcal L$ is fixed (it might be empty).
For notational convenience we set $L_i = H^1(H_i;\R)$ for $i\geqslant 0$.
A subset of $L_i$ (for $i \geqslant -m$) is \emph{very rich} \iff it is open and contains all irrational characters in $L_i$.

Now we define a notion of a rich set inductively: a subset of $L_{-m}$ is \emph{rich} \iff it is very rich. For $i>-m$, a subset $U$ of $L_i$ is \emph{rich} \iff ${\overline U}^\circ \cap L_{i-1}$ is rich in $L_{i-1}$, where ${\overline U}^\circ$ denotes the interior of the closure of $U$, all taken in $L_i$.
\end{dfn}

Clearly, every very rich subset is also rich.

Note that both notions of richness are relative to the chosen witnessing chain of $G$, and the chain $\mathcal{L}$.
The witnessing chain is fixed throughout; we also fix $\mathcal{L}$.

The reason for introducing $\mathcal L$ will became apparent in \cref{KHn in KG}. Note that we do not require the fixed dense sets inside of $L_{-i}$ to satisfy any additional properties -- these dense sets will be coming from the intrinsic notion of irrationality in linear subspaces of $H^1(G;\R)$, and they are in general not the intersections of the subspaces $L_{-i}$ with the irrational characters in $H^1(G;\R)$.

\begin{lem}
\label{rich rmk}
 The intersection of two (and hence finitely many) rich subsets of $L_i$ is rich. Thus, every rich subset of $H^1(H_i;\R)$ contains a rich subset which is $Q_i$-invariant.
\end{lem}
\begin{proof}
Clearly the intersection of two very rich sets is itself very rich. Thus the statement holds for $i=-m$. We now proceed by induction: Let $V$ and $V'$ be two open subsets of $L_i$ such that ${\overline V}^\circ \cap L_{i-1}$ and  ${\overline {V'}}^\circ \cap L_{i-1}$ are rich. By the inductive hypothesis, ${\overline V}^\circ \cap {\overline {V'}}^\circ \cap L_{i-1}$ is also rich. It is an easy exercise in point-set topology to show that for open sets $V, V'$ we have
\[
 {\overline V}^\circ \cap {\overline {V'}}^\circ = {\overline {V \cap V'}}^\circ
\]
This finishes the proof of the first part of the statement.

The second part of the statement follows by intersecting all rich sets in a $Q_i$-orbit of a given rich set.
\end{proof}

\begin{lem}
 \label{same irrationals}
 Let $U$ and $V$ be two open subsets of $L_i$ such that the intersections of $U$ and $V$ with the set of irrational characters of $L_i$ coincide. If $U$ is rich, then so is $V$.
\end{lem}
\begin{proof}
 Since $U$ is open and the set of irrational characters is dense, the closure $\overline U$ is equal to the closure of the intersection of the two. Hence $\overline U = \overline V$, and the result follows.
\end{proof}

\begin{lem}
\label{closure of rich}
Suppose that $\mathcal{L} = \emptyset$.
If $U$ is a rich subset of $H^1(H_n;\R)$, then its closure in $H^1(H_n;\R)$ contains $H^1(G;\R)$.
\end{lem}
\begin{proof}
The proof is an induction on $n$. When $n=0$ the set $U$ is in fact very rich, and hence is dense in $H^1(G;\R)$.

Now suppose that $U$ is an open subset of $H^1(H_n;\R)$ such that ${\overline U}^\circ \cap H^1(H_{n-1};\R)$ is rich.
The inductive hypothesis tells us that $H^1(G;\R)$ lies in the closure of ${\overline U}^\circ \cap H^1(H_{n-1};\R)$ taken in $H^1(H_{n-1};\R)$.  But this closure is clearly contained in $\overline{{\overline U}^\circ}$ taken in $H^1(H_n;\R)$, and obviously $\overline{{\overline U}^\circ} \subseteq \overline U$.
\end{proof}

\begin{dfn}[$\K(G,\mathcal L)$]
\label{K def}
An element  $x \in \D(G)$ is \emph{well representable} with \emph{associated integer} $n$ and \emph{associated rich sets} $\{U_i \mid i \geqslant n\}$ \iff for every $i \geqslant n$ the set $U_i$ is rich in $H^1(H_i;\R)$ and $x \in s_i ( \D(H_i,U_i) Q_i )$.

We denote the set of all well-representable elements in $\D(G)$ by $\K(G,\mathcal L)$. We write $\K(G) = \K(G,\emptyset)$.
\end{dfn}

\begin{rmk}
\label{KHn in KG}
Every $H_n$ is itself a RFRS group with witnessing chain $(H_i)_{i \geqslant n}$.
  Since we have $\D(H_n) \leqslant \D(G)$, it is immediate that every element in $\K(H_n,\mathcal L) $ is automatically an element of $\K(G, \mathcal L)$ when $\mathcal L$ is a chain of subspaces of $H^1(G;\R)$, since we may use the same associated integer and the same associated rich sets.

  Conversely, every element $x \in \K(G, \mathcal L) \cap  \D(H_n)$ lies in $\K(H_n, \mathcal L ') $,
  where
  \[\mathcal L' = \mathcal L \cup \{ H^1(H_j;\R) \mid 0 \leqslant j < n \} \]
  (and where the dense subset of irrationals in  every $H^1(H_j;\R)$ is the obvious one)
  since it is immediate that when we write $x = s_i\left( \sum_{q\in Q_i} x_q q \right)$ with $x_q \in \D(H_i,U)$ for $i \geqslant n$ and some rich set $U$, we have $x_q = 0$ for every $q \not\in \beta_i(H_n)$ (formally, this follows from the fact that $s_i \colon \D(H_i)Q_i \to \D(G)$ is an isomorphism, as shown in \cref{DHQ is DG}).
\end{rmk}

The above remark is the reason why we introduced $\mathcal L$: otherwise, we could not pass from $\K(G) \cap  \D(H_1)$ to $\K(H_1)$.

\begin{lem}
\label{KG contains QG}
For every $\mathcal L$, the subset
 $\K(G, \mathcal L)$ is a subring of $\D(G)$ containing $\Q G$.
\end{lem}
\begin{proof}
Take $x,y \in \K(G, \mathcal L)$, and let $n$ denote the maximum of their associated integers.

By \cref{rich rmk}, for every $i \geqslant n$ there exists a $Q_i$-invariant rich set $U_i$ such that
\[
 x,y \in s_i \big( \D(H_i,U_i)Q_i \big)
\]
Since $\D(H_i,U_i)Q_i$ is a ring, we have $x+y, x \cdot y \in \D(H_i,U_i)Q_i$, and hence $x+y$ and $x \cdot y$ are both well representable. This shows that $\K(G, \mathcal L)$ is a ring.

Every element in $\Q G$ is well representable with associated integer $0$ and associated (very) rich sets $U_i = H^1(H_i;\R)$.
\end{proof}

\begin{lem}
\label{main arg}
 Let $x \in \K(G, \mathcal L) \s- \{0\}$, and let $A$ be a finite indexing set. Suppose that for every $a \in A$ we are given an element $x_a \in \K(G, \mathcal L) \s- \{0\}$ with an inverse $x_a^{-1} \in \K(G, \mathcal L)$. Let $U$ be a rich subset of
 $H^1(G;\R)$.
 If for every irrational $\phi \in H^1(G;\R)$ which lies in $U$ there exists $a \in A$ with
 \[
  \phi(x_a) = \phi(x) < \phi(x-x_a)
 \]
then $x$ is invertible in $\K(G, \mathcal L)$.
\end{lem}
\begin{proof}
Take $j\in \N$.
Recall that $s_j \colon \D(H_j) Q_j \to \D(G)$ is an isomorphism.

For every $a \in A$ we set
\[
 V_{j,a} =  \Big\{ \psi \in H^1(H_j;\R) \, \Big\vert \, \val\psi {Q_j}\big({s_j}^{-1}( x-x_a ) \big) + \val \psi {Q_j}\big({s_j}^{-1}(x_a^{-1}) \big) >  2 \vert \psi \vert_{Q_j} \Big\}
\]
We observe that $V_{j,a}$ is open.
We define
\[V_j = \bigcup_{a \in A} V_{j,a}\]
and observe that $V_j$ is an open set as well.

We claim that
 $V_j$ contains all irrational characters of $H^1(G;\R)$ lying in $U$. Indeed, let $\phi$ be such a character. Since $\phi \in  H^1(G;\R)$, we have $\vert \phi \vert_{Q_j} = 0$ for all $j$ by \cref{qval for phi}.
 By assumption, there exists $a \in A$ such that
\[
 \val \phi {Q_j}\big({s_j}^{-1}(x-x_a)\big) = \phi(x-x_a) > \phi(x_a)
\]
(here we have used \cref{qval for phi} again).
Also,
\[\val \phi {Q_j}\big({s_j}^{-1}(x_a^{-1}) \big) =\phi(x_a^{-1}) =  -\phi(x_a)\]
These three facts immediately imply that $\phi \in V_{j,a}$, and hence prove the claim.
We conclude, using \cref{same irrationals}, that $V_j \cap H^1(G;\R)$ is rich in $H^1(G;\R)$, and therefore $V_j$ is rich in $H^1(H_j;\R)$ (formally, this requires an easy induction argument, which is left to the reader).
\smallskip

Let $V'_j$ be the intersection of the $Q_j$-orbit of $V_j$; the set $V'_j$ is rich by \cref{rich rmk}.
Let $n$ be the maximum of the associated integers of the elements $x,x_a$ and $x_a^{-1}$ for all $a \in A$.
For every $j \geqslant n$ there exist a $Q_j$-invariant rich set $W_j$ such that
\[x,x_a,x_a^{-1} \in s_j \big( \D(H_j, W_j)Q_j \big)\]
for every $a$ (this uses \cref{rich rmk}). We define $U_j = V'_j \cap W_j$. It is immediate that $U_j$ is rich for every $j$. We now claim that for every $j \geqslant n$, the element $x$ admits an inverse in $s_j \big( \D(H_j,U_j) Q_j \big)$; clearly this will imply that $x$ is invertible in $\K(G, \mathcal L)$, since each of the inverses will coincide with the inverse of $x$ in $\D(G)$.


Fix $j \geqslant n$, and let $\psi \in U_j$. Clearly  $\psi \in V_{j,a}$ for some $a$. Also,
\[{\psi}^{Q_j} = \{ \psi^q \mid q \in {Q_j}\} \subseteq W_j \cap V'_j\]
as both $W_j$ and $V'_j$ are $Q_j$-invariant. Therefore $x, x_a, x_a^{-1} \in s_j \big( \D(H_j, \psi^{Q_j}) Q_j \big)$, and we may
apply \cref{key lem}; 
we conclude that $x = x_a + (x-x_a)$ admits an inverse in $s_j \big( \D(H_j,\psi) Q_j \big)$. Since $\psi$ was arbitrary, and the various inverses coincide in $\D(G)$,
we conclude that $x$ is invertible in $s_j \big( \D(H_j,U_j) Q_j \big)$, and so $x$ is invertible in $\K(G, \mathcal L)$.
\end{proof}

\begin{prop}
\label{QG inv}
 Every $x \in \Q G \s- \{0\}$ is invertible in $\K(G, \mathcal L)$.
\end{prop}
\begin{proof}
The proof is an induction on $\kappa = \vert \supp x \vert$. By assumption, we have $\kappa \geqslant 1$. The base case $\kappa=1$ is immediate, since then $x$ is already invertible in $\Q G$.

\medskip
In the inductive step, we will assume that the result holds for all elements $y$ with $\vert \supp y \vert<\kappa$  across all finitely generated RFRS groups and all chains $\mathcal L$.
Recall that $\alpha \colon G \to \fab{G}$ is the free abelianisation map.
There are two cases to consider.

\case{1} Suppose that $\vert \alpha(\supp x) \vert \geqslant 2$.

We write
\[
 x = \hspace{-4mm} \sum_{a \in \alpha(\supp x)} \hspace{-4mm} x_a
\]
where for each $a$ we have $x_a \in \Q G$ with $\alpha(\supp x_a) = \{a\}$.

Using the inductive hypothesis we see that every $x_a$ is invertible in $\K(G, \mathcal L)$.
Now we apply \cref{main arg} with $U = H^1(G;\R)$ and $A =\alpha(\supp x)$, and where to each irrational $\phi$ we associate $a \in A$ on which $\phi\vert_A$ is minimal (note that such an $a$ is unique, since $\phi$ is irrational).

\case{2} Suppose that $\kappa >1$ but $\vert \alpha(\supp x) \vert = 1$.
In this case there exists $\gamma \in G$ and $i$ such that $\gamma x \in \Q H_{i}$ but $\gamma' x \not\in \Q H_{i+1}$ for every $\gamma' \in G$ (we are using here that $(H_j)_j$ is a residual chain). It is clear that without loss of generality we may assume that $\gamma=1$.

Let $\alpha_i \colon H_{i} \to \fab{H_{i}}$ be the free abelianisation map. If $\alpha_i(\supp x)$ is supported on a singleton, then there exists $\gamma'$ such that $\gamma' x$ is supported on $\ker \alpha_i$. But $\ker \alpha \leqslant H_{i+1}$, which contradicts our assumption on $i$. Thus
\[
 \vert\alpha_i(\supp x) \vert \geqslant 2
\]
We now apply the argument of Case 1 to $x$ thought of as an element of $\Q H_{i}$, and conclude that $x$ admits an inverse in $\K(H_i, \mathcal L')$, where $\mathcal L' = \mathcal L \cup \{H_1(H_j;\R) \mid j<i \}$. But $\K(H_i, \mathcal L')$ is a subring of $\K(G, \mathcal L)$ by \cref{KHn in KG}, and this finishes the proof.
\end{proof}

\begin{lem}
\label{sets Va}
Let $S \subseteq H^1(G;\R)$ and let $x \in \D(G,S) \s- \{0\}$ be given. There exists a very rich subset $V \subseteq H^1(G;\R)$ such that
 \[
  V = \bigsqcup_{a \in A} V_a
 \]
where $A$ is a finite indexing set, and where for every $a \in A$ there is $x_a \in \Q G \s- \{0\}$ such that for every $\phi \in V_a \cap S$ we have
\[
 \phi(x_a) = \phi(x) < \phi(x-x_a)
\]
\end{lem}
\begin{proof}
Recall that $\alpha \colon G \to \fab{G}$ denotes the free abelianisation map; let $K$ denote its kernel, and let $s_\alpha$ denote a section.

 We have $x = s_\alpha(y) s_\alpha(z)^{-1}$ with $y,z \in \D(K) \fab{G} \s- \{0\}$ (since $x \neq 0$).
We define $A =  \supp y \times \supp z$ where the supports are subsets of $\fab{G}$,  
and for $a = (b,c)$ we declare $V_a$ to be the set of those characters in $H^1(G;R)$ which attain their minima on $\supp y$ precisely at $b$, and on $\supp z$ precisely at $c$.
It is immediate that the sets $V_a$ are open and pairwise disjoint. It is also immediate that $V = \bigsqcup_{a \in A} V_a$ is very rich, since it contains all irrational characters in $H^1(G;R)$.

Now let us fix $a = (b,c) \in A$. We write
\[
 y = y_b + y' \textrm{ and } z = z_c + z'
\]
where $y_b, y', z_c, z' \in \D(K)\fab{G}$ satisfy $\supp y_b = \{b\}$, $\supp y' = \supp y \s- \{b\}$, $\supp z_c = \{c\}$, and $\supp z' = \supp z \s- \{c\}$.
Take $\phi \in S \cap V_a$. By the definition of $\D(G,S)$, there exists a biordering $\leqslant$ on $\fab{G}$ compatible with $\phi$ such that $i_\leqslant(x) \in \nov {(\Q K)} {\fab{G}} \phi$. Now
\[
 i_\leqslant(x) = i_\leqslant(s_\alpha(y))i_\leqslant(s_\alpha(z))^{-1} = y \id_\leqslant(z)^{-1}
\]
and so
\[
 i_\leqslant(x)  z =  y
\]
where both equations hold in $\F_\leqslant(G)$.

Applying $\phi$ to both sides of the equality tells us that
 $\phi$ attains its minimum on $\supp \iota_\leqslant(x)$ at $d$ which satisfies $dc = b$. Note that this forces $d$ to be unique.
 Let
\[
 x'_a = i_\leqslant(x)(d) d \in (\Q K) \fab{G}
\]
We immediately see that $\phi(x'_a) < \phi\big( i_\leqslant(x) - x'_a\big)$. Setting $x_a = s_\alpha(x'_a) \in \Q G$, we obtain
\[
 \phi(x-x_a) = \phi\big(i_\leqslant(x-x_a)\big) = \phi\big( i_\leqslant(x) - x'_a\big) > \phi(x'_a) = \phi(x_a)
\]

Note that $\iota_\leqslant(x)(d) = y(b) z(c)^{-1}$ is independent of the choice of $\leqslant$. Therefore $x_a$ depends only on $x$ and $a \in A$, but not on $\phi$.
\end{proof}

\begin{thm}
\label{DG is KG}
Let $G$ be a finitely generated RFRS group. We have
 \[\K(G, \mathcal L) = \D(G)\]
 for every chain $\mathcal L$.
\end{thm}
\begin{proof}
Recall that  every sub-skew-field of $\D(G)$ containing $\Q G$ is equal to $\D(G)$ by \cref{subfield of D}.
In view of this fact, we need only show that $\K(G, \mathcal L)$ is a skew-field (since we have already seen in \cref{KG contains QG} that $\K(G, \mathcal L)$ contains $\Q G$). To this end, take $x \in \K(G, \mathcal L) \s- \{0\}$. 
Let $n$ denote the associated integer of $x$. We proceed by induction on $n$.

\smallskip
\noindent \textbf{Base case:} Suppose that $n=0$. Let $U$ denote the associated rich set of $x$ in $H^1(G;\R)$. 
We now apply \cref{sets Va} to $x \in \D(G,U)$, which produces for us a very rich set $V = \bigsqcup_{a \in A} V_a$ where $A$ is a finite set. 

Let $\phi \in H^1(G;\R)$ be an irrational character contained in $U$. \cref{sets Va} tells us that there exists $a \in A$ and $x_a \in \Q G \s- \{0\}$ such that
\[
 \phi(x_a) = \phi(x) < \phi(x-x_a)
\]
Also, since $x_a \in \Q G \s- \{0\}$, it is invertible in $\K(G, \mathcal L)$ by \cref{QG inv}. Now an application of \cref{main arg} gives us invertibility of $x$ in $\K(G, \mathcal L)$.

\smallskip
\noindent \textbf{Inductive step:} Let
\[
 x = s_n\left( \sum_{q\in Q_n} x_q q \right)
\]
with $x_q \in \D(H_n)$ for every $q \in Q_n$.

Let $B = \{ q \in Q_n \mid x_q \neq 0\}$,
and consider the map $\xi \colon B \to Q_1$ given by $\xi(q) = \beta_1 s_n(q)$.
If $\xi$ is constant on $B$, then up to multiplying $x$ by an element of $G$ we may assume that $\xi(B) = \{1\}$. This implies that $s_n(q) \in H_1$ for every $q \in B$, and so $x = s_n(\sum x_q q) \in \D(H_1)$.
Therefore $x \in \K(H_1, \mathcal L')$ (using \cref{KHn in KG}), where
\[
 \mathcal L' = \mathcal L \cup \{ H^1(G;\R) \}
\]
Treated as an element in $\K(H_1, \mathcal L')$, the element $x$ has lower associated integer, and hence is invertible in $\K(H_1, \mathcal L')$ (and therefore in $\K(G, \mathcal L)$) by induction.

Now suppose that $\xi$ is not constant on $B$.
Let $A = \{ \xi(q)) \mid q \in B \}$.
 Clearly, $A$ is a finite set. For every $a \in A$, the element
 \[x_a =  s_n \left( \sum_{\xi(q) =a} x_q q \right)\]
 is invertible in $\K(G, \mathcal L)$ by the above discussion.
 Also,
 \[
  x_a s_1(a)^{-1} = \sum_{\xi(q) = a} x_q s_n(q) s_1(a)^{-1} \in \D(H_1)
 \]
since $x_q \in \D(H_n) \leqslant \D(H_1)$ and $s_n(q) s_1(a)^{-1} \in H_1$.

Consider the restriction $\alpha \colon H_1 \to \fab{G}$. Its image $\alpha(H_1)$ is a finitely generated free-abelian group, and its kernel is $K$ (which is also the kernel of $\alpha \colon G \to \fab{G}$). Thus, \cref{DG is Ore} tells us that the isomorphism 
\[\Ore(s_\alpha) \colon \Ore(\D(G) \fab{G}) \to \D(G)\]
 restricts to an isomorphism
\[
 \Ore(\D(K) \alpha(H_1)) \to \D(H_1)
\]
Hence,
\[
 \phi(x_a) = \phi\big( x_a s_1(a)^{-1} \big) + \phi(s_1(a)) \in \phi(s_1(a)H_1) \subseteq \R
\]
for every $\phi \in H^1(G;\R)$.

Let $\phi \in H^1(G;\R)$ be irrational, and take distinct $a$ and $b$ in $A$. Since $\phi$ is injective as a map $\fab{G} \to \R$, the set of values that $\phi$ attains at the coset $s_1(a) H_1$ of $H_1$ in $G$ is disjoint from the set of values it attains at $s_1(b) H_1$ -- here we are using the fact that $H_1 \geqslant \ker \alpha$ in a crucial way.
Therefore
\[
 \phi(x_a) \neq   \phi(x_b)
\]
Hence, for each irrational $\phi$ there exists $a \in A$ such that
\[
 \phi(x) = \phi( x_a)  < \phi(x - x_a)
\]
We finish the argument by an application of \cref{main arg}.
\end{proof}

\section{The main results}

Throughout this section, all unspecified tensoring happens over $\Q G$.

\begin{dfn}[Antipodal map]
 Given a vector space $V$, we call the map $V \to V$ given by $v \mapsto -v$ the \emph{antipodal map}.
\end{dfn}

\begin{thm}
\label{main thm chain cplxs}
Let $G$ be a finitely generated RFRS group, and let $N \in \N$ be an integer.
Let $C_*$ denote a chain complex of free $\Q G$-modules such that for every $i \leqslant N$ the module $C_i$ is finitely generated, and such that $H_i(\D(G) \otimes C_\ast) = 0$.
 There exists a finite-index subgroup $H$ of $G$ and an open subset $U \subseteq H^1(H;\R)$ such that
\begin{enumerate}
 \item the closure of $U$ contains $H^1(G;\R)$;
 \item $U$ is invariant under the antipodal map;
 \item $H_i\big(  \novq H \psi \otimes_{\Q H} C_\ast \big) =0$ for every $i \leqslant N$ and every $\psi \in U$.
\end{enumerate}
\end{thm}
\begin{proof}
We fix a free $\Q G$-basis of every $C_i$ with $i \leqslant N$. Since we only care about homology up to degree $N$, we apply the following procedure to $C_\ast$: all modules in degrees above $N+1$ are set to $0$; since $H_N(\D(G) \otimes C_\ast) = 0$ and $C_N$ is finitely generated, there exists a finitely generated free submodule $\overline{C_{N+1}}$ of $C_{N+1}$ such that replacing $C_{N+1}$ by $\overline{C_{N+1}}$ still yields vanishing of the $N^{th}$ homology with coefficients in $\D(G)$. We replace $C_{N+1}$ by $\overline{C_{N+1}}$.
We will continue to denote the new chain complex by $C_\ast$ -- it is now a finite chain complex of finitely generated free modules, and we still have
\[
H_i(\D(G) \otimes C_\ast) = 0
\]
for every $i \leqslant N$.
The homology of the new complex $C_\ast$ with any coefficients agrees with the homology of the old complex in degrees lower than $N$; in degree $N$ this does not have to be the case, but nevertheless if some homology of the new complex vanishes in degree $N$, then it must have vanished for the old complex too.

By assumption, we have $H_i(\D(G) \otimes C_\ast) = 0$ for all $i \leqslant N$. Since $\D(G)$ is a skew-field, this implies the existence of invertible matrices $M_0, \dots, M_{N+1}$ over $\D(G)$ such that if we change the basis of $\D(G) \otimes C_i$ by $M_i$, we obtain a chain complex $C'_\ast$ which can be written as follows: every $C'_i$ 
splits as $D_i \oplus E_i$, and every differential is the identity matrix taking $D_{i+1}$ isomorphically to $E_i$, and is trivial on $E_{i+1}$.

 We have shown in \cref{DG is KG} that $\D(G) = \K(G)$. Since we are looking at finitely many matrices $M_i$ and their inverses, and each of them is finite, there are only finitely many entries appearing in these matrices. Therefore, there exists $n$ (the maximum of the associated integers) and a rich set $U \subseteq H^1(H_n;\R)$ (where $(H_i)_i$ denotes a witnessing chain of $G$) such that all the matrices $M_i$ and their inverses lie over $\D(H_n, U) Q_n$, where $Q_n = G /H_n$. We additionally assume that $U$ is $Q_n$-invariant and invariant under the antipodal map (see \cref{rich rmk}).
 For notational convenience we set $H = H_n$ and $Q = Q_n$. 
 Note that the closure of $U$ contains $H^1(G;\R)$ by \cref{closure of rich}. 


 Every entry of a matrix $M_i$ lies in $\D(H,U) Q$. The ring $\D(H, U) Q$ has the structure of a free finitely generated left $\D(H, U)$-module. Therefore the right action of $M_i$ on the free $\D(G)$-module $\D(G)\otimes C_i$ can be seen as right-multiplication by a matrix $M'_i$ with entries in $\D(H,U)$ on the free $\D(H,U)$-module $ \D(H,U) \otimes_{\Q H} C_i$.

 Take $\psi \in U$. By the very definition of $\D(H,U)$, we may replace each entry of $M'_i$ by an element of $\novq H \psi$. Applying this procedure to the inverse of the matrix $M_i$ gives a matrix over $\novq H \psi$ which is the inverse of $M'_i$; we will therefore denote it by ${M_i'}^{-1}$.

 We are now ready to conclude the proof: for every $i$ there exists a matrix $M'_i$ over $\novq H \psi$ invertible over the same ring, and such that if we change the basis of $ \novq H \psi \otimes C_\ast$ using these matrices, we obtain a chain complex $C''_\ast$ such that
 every $C''_i$ splits as $D'_i \oplus E'_i$, and every differential vanishes on $E_i'$ and takes $D'_{i+1}$ isomorphically to $E'_i$. The homology of this chain complex obviously vanishes in every degree.
\end{proof}

\begin{thm}
\label{main statement}
  Let $G$ be an infinite finitely generated group which is virtually RFRS. Then $G$ is virtually fibred, in the sense that it admits a finite-index subgroup mapping onto $\Z$ with a finitely generated kernel, \iff $\beta^{(2)}_1(G) = 0$.
\end{thm}
\begin{proof}
 Suppose first that $G$ is virtually fibred; let $H$ denote a finite index subgroup of $G$ which maps onto $\Z$ with a finitely generated kernel. By \cref{Lueck}, the first $L^2$-Betti number of $H$ vanishes. But \cite{Lueck2002}*{Theorem 1.35(9)} tells us that then $\beta^{(2)}_1(G) = 0$ as well.

 \smallskip
 Now let us look at the other direction. Suppose that $\beta^{(2)}_1(G) = 0$. Note that the first $L^2$-Betti number of every finite index subgroup of $G$ vanishes as well (by the same argument as above), and being virtually fibred clearly passes to finite index overgroups, and therefore we may assume that $G$ itself is RFRS, infinite and finitely generated. Since $G$ is infinite, we have $\beta_0^{(2)}(G) = 0$ by \cite{Lueck2002}*{Theorem 6.54(8b)}.

 We apply \cref{main thm chain cplxs} where $C_\ast$ is constructed as follows: take the cellular chain complex of the universal covering of some classifying space of $G$ with finite $1$-skeleton, and  tensor it with $\Q G$. We take $N =1$. We conclude the existence of a finite index subgroup $H$ and a non-empty open subset $U \subseteq H^1(H;\R)$ such that
 \[
   H_1(  \novq H \psi \otimes_{\Q H} C_\ast) = 0
 \]
for every $\psi \in U$.
Let us take $\psi \in U$ whose image is $\Z$ -- such a $\psi$ exists since $U$ is open, non-empty, and $H^1(H;\R)$ is non-trivial. Since $\Q H \otimes_{\Q H} C_\ast$ is a resolution of the trivial $\Q H$-module $\Q$, the vanishing of the above homology tells us that
\[
  H_1(H; \novq H \psi) = 0
\]
Since $U$ is invariant under the antipodal map, we also have
\[
  H_1(H; \novq H {-\psi}) = 0
\]
Now Sikorav's Theorem (\cref{Sikorav}) tells us that $\{\psi, -\psi\} \subseteq \Sigma(H)$, and so $\ker \psi$ is finitely generated by \cref{BNS fibring}. This finishes the proof.
\end{proof}

Recall that a group $G$ is said to be \emph{of type $\typeFP{2}$} \iff there exists a long exact sequence
\[
 \dots \to C_2 \to C_1 \to C_0 \to \Z
\]
of projective $\Z G$-modules $C_i$, where $\Z$ is considered to be a trivial $\Z G$-module, such that $C_2, C_1$ and $C_0$ are finitely generated. By standard homological algebra (see for example \cite{Bieri1981}*{Remark 1.1(2)}) we may in fact take the modules $C_i$ to be free for all $i$, and the modules $C_2, C_1,$ and $C_0$ to be additionally finitely generated.

Note that being of type $\typeFP{2}$ forces $G$ to be finitely generated.

A group $G$ is \emph{of cohomological dimension at most $2$} \iff there exists an exact sequence
\[
 0 \to D_2 \to D_1 \to D_0 \to \Z
\]
of projective $\Z G$-modules $D_i$, where $\Z$ is again a trivial $\Z G$-module.

Again, standard homological algebra (or \cite{Bieri1981}*{Proposition 4.1(b)}) tells us that if $G$ is both of type $\typeFP{2}$ and of cohomological dimension at most $2$, we may take $D_2$ to be finitely generated and projective, and $D_1$ and $D_0$ to be finitely generated and free.

The author is grateful to the referee for pointing out the following application.

\begin{thm}
\label{thm cohdim 2}
 Let $G$ be a non-trivial virtually RFRS group of type $\typeFP 2$ and of cohomological dimension at most $2$. If $\beta_2^{(2)}(G) = \beta_1^{(2)}(G)= 0$, then there exists a finite index subgroup $H$ of $G$ and an epimorphism $\psi \colon H \to \Z$ with kernel of type $\typeFP 2$.
\end{thm}
\begin{proof}
Since $G$ is of type $\typeFP{2}$, we have an exact sequence $C_\ast$ as above.
Observe that
$H_i(\D(G) \otimes_{\Z G} C_\ast) = 0$ for all $i\leqslant 2$, since the $L^2$-Betti numbers of $G$ vanish -- note that $\beta_0^{(2)}(G) = 0$ as $G$ is non-trivial and of finite cohomological dimension, which together imply that $G$ is infinite, which suffices by \cite{Lueck2002}*{Theorem 6.54(8b)}. In fact, the higher $L^2$-Betti numbers also vanish, as can be easily seen using the resolution $D_\ast$ from above.

Note that being infinite implies that $\fab{G}$ is non-trivial.

Now, the chain complex $C'_\ast = \Q G \otimes_{\Z G} C_\ast$ satisfies the assumptions of \cref{main thm chain cplxs} with $N=2$. From the theorem we obtain a finite index subgroup $H$, and a surjective character $\psi \colon H \to \Z$ such that
\[
 H_i\big(  \novq H \psi \otimes_{\Q H} C'_\ast \big) =0 = H_i\big(  \novq H {-\psi} \otimes_{\Q H} C'_\ast \big)
\]
for $i \in \{0,1,2\}$.

Since $\psi$ is non-trivial, we have
\[
H_0\big(  H; \nov \Z H \psi \big) =0 = H_0\big(  H; \nov \Z H {-\psi} \big)
\]
(this is an easy exercise).

By \cref{Sikorav}, we have
\[
H_1\big(  H; \nov \Z H \psi \big) =0 = H_1\big(  H; \nov \Z H {-\psi} \big)
\]
We now want to compute the second homology of $H$ with these coefficients. Consider the resolution $D_\ast$ from above. We have $D_3 = 0$ and
\[
 H_2( \novq H \psi \otimes_{\Z H} D_\ast) = H_2( H; \novq H \psi) = 0
\]
This implies that $ \novq H \psi \otimes_{\Z H} \partial$ is injective, where $\partial \colon D_2 \to D_1$ denotes the differential.

The module $D_2$ is finitely generated and projective, and so we have an isomorphism of $\Z H$-modules
\[
 D_2 \oplus E \cong \Z H^n
\]
for some $n$ and some $\Z H$-module $E$. We immediately see that
\[
 \left( \nov \Z H \psi \otimes_{\Z H} D_2 \right) \oplus \left( \nov \Z H \psi \otimes_{\Z H} E \right) \cong \nov \Z H \psi \otimes_{\Z H} \Z H^n = \left( \nov \Z H \psi \right)^n
\]
and
\[
 \left( \novq H \psi \otimes_{\Z H} D_2 \right) \oplus \left( \novq H \psi \otimes_{\Z H} E \right) \cong \left( \novq H \psi \right)^n
\]
Since the natural map $\nov \Z H \psi \to \novq H \psi$ is clearly injective, the same is true for the induced map
\[
 \left( \nov \Z H \psi\right)^n \to \left( \novq H \psi \right)^n
\]
and hence for the natural map
\[
 \nov \Z H \psi \otimes_{\Z H} D_2 \to \novq H \psi \otimes_{\Z H} D_2
\]
This implies that $\nov \Z H \psi \otimes_{\Z H} \partial$ is injective, and so
\[
 H_2( H; \nov \Z H \psi) = H_2( \nov \Z H \psi \otimes_{\Z H} D'_\ast) =  0
\]
A completely analogous argument shows that
\[
 H_2( H; \nov \Z H {-\psi}) = 0
\]
also holds. We now conclude that $\ker \psi$ is of type $\typeFP{2}$ from a theorem of Pascal Schweitzer \cite{Bieri2007}*{Theorem A.1} (which should be thought of as a higher degree version of Sikorav's theorem).
\end{proof}

\section{Agol's Theorem}

In~\cite{Agol2008}, Agol proved the following theorem.

\begin{thm}[Agol's theorem]
 Let $M$ be a compact connected orientable irreducible $3$-manifold with $\chi(M) = 0$
such that $\pi_1(M)$ is RFRS. If $\phi \in H^1(M;\Z) \s- \{0\}$ is a non-fibred homology class, then there
exists a finite-sheeted cover $p \colon M' \to M$ such that $p^\ast  \phi \in H^1 (M';\Z)$ lies in the cone over
the boundary of a fibred face of $B(M')$.
\end{thm}

Let us explain the notation: $\chi$ stands for the Euler characteristic. A  non-trivial character $\phi \colon \pi_1(M) \to \Z$ is \emph{fibred} \iff it is induced by a fibration of $M$ over the circle; equivalently (thanks to a theorem of Stallings~\cite{Stallings1962}), $\phi$ is fibred \iff its kernel is finitely generated. The notation $p^\ast \phi$ corresponds to simply $\phi$ in our notation, since we identify $H^1(M;\R)$ with a subspace of $H^1(M';\R)$.

The notation $B(M)$ stands for the Thurston polytope (defined in \cite{Thurston1986}). It is a compact rational polytope lying in $H^1(M;\R)$, and its key property is that certain maximal open faces of $B(M)$ are called fibred, and a primitive integral character $\phi \in H^1(M;\Z) \s- \{0\}$ is fibred \iff it lies in the cone in $H^1(M;\R)$ over a fibred face. (This property of $B(M)$ was shown by Thurston in \cite{Thurston1986}, and then reproved by the author in \cite{Kielak2018a} using methods very similar to the ones used in this article).

We will now give a new proof of Agol's theorem; in fact, we give a slightly stronger, more uniform statement, since we will show that the finite cover $M'$ can be chosen independently of $\phi$.

\begin{thm}[Uniform Agol's theorem]
\label{agol uniform}
 Let $M$ be a compact connected orientable irreducible $3$-manifold with $\chi(M) = 0$
such that $\pi_1(M)$ is RFRS. There
exists a finite-sheeted cover $p \colon M' \to M$ such that for every $\phi \in H^1(M;\Z) \s- \{0\}$ either $\phi$ is fibred or  $p^\ast \phi$ lies in the cone over
the boundary of a fibred face of $B(M')$.
\end{thm}
\begin{proof}
 We begin by observing that an irreducible $3$-manifold with infinite fundamental group is aspherical. Since the statement is vacuous when $G = \pi_1(M) = \{1\}$, and since $G$ is RFRS, we may assume that $G$ is infinite, and thus that $M$ is as\-pher\-i\-cal.
Now $\chi(M) = 0$ is equivalent to the vanishing of the $L^2$-Betti numbers by \cite{LottLueck1995}*{Theorem 0.1b}.

Since $M$ is compact, $G$ is finitely generated. We may now apply \cref{main thm chain cplxs} with $N=3$ to the cellular chain complex $C_\ast$ of the universal covering of $M$ (with $\Q$-coefficients). We are given a finite index subgroup $H \leqslant G$ and an open set $U$ in $H^1(H;\R)$ which is invariant under the antipodal map, and whose closure contains $H^1(G;\R)$. Furthermore, we have
\[
 H_1\big( \novq H \psi \otimes_{\Q H} C_\ast  \big) =0
\]
for every $\psi \in U$.

Let $M' \to M$ be the covering corresponding to $H \leqslant G$. Then $\Q H \otimes_{\Q H}C_\ast$ coincides with the cellular chain complex of the universal covering of $M'$ (as a $\Q H$-module).

Let $\phi \colon G \to \Z$ be non-trivial. Since $\phi \in H^1(G;\R)$, there exists a sequence $(\psi_i)_i$ of characters in $U$ with
\[\lim_{i \longrightarrow \infty} \psi_i =  \phi\]
Since $U$ is invariant under the antipodal map, we have $- \psi_i \in U$ for every $i$, and therefore
\[
 H_1\big( H; \novq H {\psi_i} \big) = H_1\big( H; \novq H {-\psi_i} \big) = 0
\]
for every $i$. Sikorav's Theorem (\cref{Sikorav}) tells us that \[\{\psi_i, -\psi_i \mid i \in \N \}\subseteq \Sigma(H)\]
and \cref{BNS fibring} tells us that each $\psi_i$ has a finitely generated kernel. Stallings's Theorem~\cite{Stallings1962} now implies that every $\psi_i$ is fibred (note that $M'$ is still irreducible). Thus, every $\psi_i$ lies in the cone of a fibred face of $B(M')$.

Since the polytope $B(M')$ has finitely many faces, and so finitely many fibred faces, we may pass to a subsequence and conclude that every $\psi_i$ lies in the cone of the same fibred face of $B(M')$. Hence, $\phi$ itself must lie in the closure of this cone. Now one of two possibilities occurs: $\phi$ lies in the cone of a fibred face, and so is itself fibred, or $\phi$ lies in the cone of the boundary of a fibred face of $B(M')$.
\end{proof}

\bibliography{bibliography}

\bigskip
\noindent Dawid Kielak \newline
\href{mailto:dkielak@math.uni-bielefeld.de?subject=RFRS groups and virtual fibring}{\texttt{dkielak@math.uni-bielefeld.de}} \newline
Fakult\"at f\"ur Mathematik  \newline
Universit\"at Bielefeld \newline
Postfach 100131  \newline
D-33501 Bielefeld \newline
Germany \newline

\end{document}